\documentclass[11pt]{amsart}
\usepackage{fullpage,amsfonts,amsmath,amssymb}
\usepackage{fancyhdr}
\usepackage{amsfonts,amsmath,amssymb,fullpage,graphicx}
\usepackage{verbatim}
\usepackage{hyperref}

\newcommand{\what}{\widehat}%
\newcommand{\wtilde}{\widetilde}%
\newcommand{\R}{\mathbb R}%
\newcommand{\C}{\mathbb C}%
\newcommand{\Z}{\mathbb Z}%
\newcommand{\N}{\mathbb N}%
\newcommand{\mfk}{\mathfrak{k}}
\newcommand{\mfp}{\mathfrak{p}}
\newtheorem{theorem}{Theorem}[section]
\newtheorem{lemma}[theorem]{Lemma}
\newtheorem{proposition}[theorem]{Proposition}
\newtheorem{corollary}[theorem]{Corollary}

\theoremstyle{definition}
\newtheorem{definition}[theorem]{Definition}

\theoremstyle{definition}
\newtheorem{remark}[theorem]{Remark}

\numberwithin{equation}{section}

\begin{document}
\title[Ramanujan's master theorem]{Ramanujan's master theorem for radial sections of line bundle over noncompact symmetric spaces}


\author{ Sanjoy Pusti and Swagato K Ray}
\address{Department of Mathematics, IIT Bombay, Powai, Mumbai-400076, India}
\email{sanjoy@math.iitb.ac.in}

\address{Stat-Math Unit, Indian Statistical Institute, 203, B.T. Road, Kolkata-700108, India}
\email{swagato@isical.ac.in}

\thanks{The first author is partially supported by the seed grant no. 17IRCCSG013 of IIT bombay}
\begin{abstract}
We prove analogues of Ramanujan's Master theorem for the radial sections of the line bundles over the Poincar\'{e} upper half plane $\mathrm{SL}(2, \R)/\mathrm{SO}(2)$ and over the complex hyperbolic spaces $\mathrm {SU}(1, n)/ S(\mathrm{U}(1) \times \mathrm{U}(n))$.
\end{abstract}
\maketitle

\section{Introduction}
Ramanujan's Master theorem (\cite{Hardy}) states that if a function $f$ can be expanded around zero in a power series of the form $$f(x)=\sum_{k=0}^\infty (-1)^k a(k) x^k,$$ then 
\begin{equation}\label{eqn-11}
\int_0^\infty f(x) x^{-\lambda-1}\,dx=-\frac{\pi}{\sin\pi\lambda} a(\lambda), \, \text{ for }\lambda\in\C.
\end{equation}

One needs some assumptions on the function $a$, as the theorem is not true for $a(\lambda)=\sin\pi\lambda$. Hardy provides a rigorous statement of the theorem above as: Let $A, p, \delta$ be real constants such that $A<\pi$ and $0<\delta\leq 1$. Let $\mathcal H(\delta)=\{\lambda\in\C\mid \Re\lambda>-\delta\}$. Let $\mathcal H(A,p,\delta)$ be the collection of all holomorphic functions $a:\mathcal H(\delta)\rightarrow \mathbb C$ such that  $$ |a(\lambda)|\leq C e^{-p (\Re\lambda) + A |\Im\lambda|} \text{ for all } \lambda\in \mathcal H(\delta),$$ where $\Re\lambda, \Im\lambda$ respectively denote the real and imaginary parts of $\lambda$. 
\begin{theorem}[Ramanujan's Master theorem, Hardy \cite{Hardy}] \label{master-theorem}Suppose $a\in\mathcal H(A,p,\delta)$. Then 
\begin{enumerate}
\item The power series \begin{equation}\nonumber 
f(x)=\sum_{k=0}^\infty (-1)^k a(k) x^k,
\end{equation} converges for $0<x<e^p$ and defines a real analytic function on that domain.
\item Let $0<\eta<\delta$. For $0<x<e^p$ we have $$f(x)=\frac{1}{2\pi i}\int_{-\eta-i\infty}^{-\eta + i\infty}\frac{-\pi}{\sin \pi\lambda} a(\lambda) x^\lambda\, d\lambda.$$The integral on the right side of the equation above converges uniformly on compact subsets of $[0, \infty)$ and is independent of $\eta$.

\item Also \begin{equation}\nonumber \int_0^\infty f(x) x^{-\lambda-1}\,dx=-\frac{\pi}{\sin \pi\lambda} a(\lambda),\end{equation} holds for the extension of $f$ to $[0, \infty)$ and for all $\lambda\in\mathbb C$ with $0<\Re\lambda<\delta$.
\end{enumerate}

\end{theorem}
This theorem can be thought of as an interpolation theorem, which reconstructs the values of $a(\lambda)$ from its given values at $a(k), k\in \N\cup \{0\}$. In particular if $a(k)=0$ for all $k\in \N\cup \{0\}$, then $a$ is identically $0$.

Bertram (in \cite{Bertram}) provides a group theoretic interpretation of the theorem in the following way: Consider $x^\lambda, \lambda\in\C$ and $x^k, k\in\Z$ as the spherical functions on $X_G=\R^+$ and $X_U=U(1)$ respectively. Both $X_G$ and $X_U$ can be realized as the real forms of their complexification $X_\C=\C^\ast$. Let $\wtilde{f}$ and $\what{f}$ denote the spherical transformation of $f$ on $X_G$ and on $X_U$ respectively. Then it follows from equation (\ref{eqn-11}) that $$\wtilde{f}(\lambda)=-\frac{\pi}{\sin\pi\lambda} a(\lambda), \hspace{.2in} \what{f}(k)=(-1)^k a(k).$$ Using the duality between $X_U=U/K$ and $X_G=G/K$ inside their complexification $X_\C=G_\C/K_\C$, Bertram proved an analogue of Ramanujan's Master theorem for Riemannian symmetric spaces of noncompact type with rank one. This theorem was further extended by \'{O}lafsson and Pasquale (see \cite{Olafsson-1}) to higher rank case. Also it was further extended for the hypergeometric Fourier transform associated to root systems by \'{O}lafsson and Pasquale (see \cite{Olafsson-2}).

In this paper we prove analogue of  Ramanujan's Master theorem for the radial sections of line bundles over the Poincar\'{e} upper half plane  and over complex hyperbolic spaces. More precisely, in the first part of the paper, we work with the group $G=\mathrm{SL}(2, \R)$ but instead of the $K$-biinvariant functions with $K=\mathrm{SO}(2)$  we consider functions with the property \begin{equation}
\nonumber
f(k_\theta g k_\alpha)=e^{in(\theta + \alpha)} f(g),
\end{equation}
where $n\in\Z$ is fixed and $k_\theta, k_\alpha\in K$. However, this case offers a fresh challenge, in the sense that one needs to take the discrete series into the consideration for $n$ positive.

In the second  part of this paper, we consider $G=\mathrm {SU}(1, n), K=S(\mathrm{U}(1) \times \mathrm{U}(n))$ and consider one dimensional representations of $S(\mathrm{U}(1) \times \mathrm{U}(n))$. The one dimensional representations of $S(\mathrm{U}(1) \times \mathrm{U}(n))$ are given by $\chi_l$ for $l\in\Z$ (see section 3 for precise definition). This is the only case among the class of  real rank one semisimple Lie groups for which nontrivial one dimensional representations of $K$ exists. We prove analogue of Theorem \ref{master-theorem} for all $\chi_l$-radial functions with the restriction that $|l|<n$ (see Definition \ref{defn-chi_l-radial} for the defintion of $\chi_l$-radial functions). The reason for working under the above restriction on $l$ is the fact that the mathematical machinery related to Ramanujan's master theorem is available only for $l\in\Z$ with $|l|<n$ (see \cite{Heckman}, \cite{Ho}). We observe that discrete series representations do not arise in the spherical inversion formula for $\chi_l$-radial functions with $|l|<n$ (see (\ref{inversion-chi_l})).  We also observe that these two cases (that is, the case of line bundles over the Poincar\'{e} upper half plane and over complex hyperbolic spaces) are mutually disjoint.

In section $2$ we prove analogue of Ramanujan's master theorem for the radial sections of line bundles over Poincar\'{e} upper half plane and in section $3$ we prove the theorem for the radial sections of line bundles over complex hyperbolic spaces.

\section{Ramujan's Master theorem for Poincar\'{e} upper half plane}

Let $G=\mathrm{SL}(2,\R)$. Let $$k_\theta=
\left(\begin{array}{lll}
\cos\theta & \sin\theta \\
-\sin\theta & \cos\theta\end{array}\right), a_t=\left(\begin{array}{lll}e^t & 0\\
0 & e^{-t}\end{array}\right) \text{ and } n_\xi=\left(\begin{array}{lll} 1 & \xi\\
0 & 1\end{array}\right).$$ Then $K=\{k_\theta\mid \theta\in[0, 2\pi)\}, A=\{a_t\mid t\in \R\}, N=\{n_\xi\mid \xi\in\R\}$ are subgroups of $G$, in which $K=\mathrm{SO}(2)$ is a maximal compact subgroup of $G$. Let $G=KAN$ be an Iwasawa decomposition of $G$ and for $x\in G$, let $x=k_\theta a_t n_\xi$ be the corresponding decomposition. Then we write $K(x):=k_\theta$, $H(x)=t$, $a(x)=a_t$ and $n(x)=n_\xi$. In fact if $x=\left(\begin{array}{lll}
a & b \\

c & d

\end{array}\right)\in \mathrm{SL}(2,\R)$, then $\theta, t$ and $\xi$ are given by \begin{equation}\label{Iwasawa} e^{2t}=a^2 + c^2, e^{i\theta}=\frac{a-ic}{\sqrt{a^2 + c^2}} \text{ and } \xi=\frac{ab +cd}{\sqrt{a^2 + c^2}}.
\end{equation}
The Haar measure $dx$ of $G$ splits according to this decomposition as $$dx=e^{2t} dk_\theta\, dt\, d\xi,$$ where $dk_\theta=(2\pi)^{-1} d\theta$ is the normalised Haar measure of $K$ and $d\xi$, $dt$  are Lebesgue measures on $\R$. The Cartan decompositon of $G$ is given by $G=K \overline{A^+} K$ where $A^+=\{a_t\mid t>0\}$ and $x\in G$ can be written by $x=k_\theta a_t k_\phi$ with $t\geq 0$. Also the Haar measure $dx$ of $G$ corresponding to that decomposition is given by $$dx=\sinh 2t\,dk_\theta\,dt\,dk_\phi.$$ Let $\sigma(x)=\sigma(k_\theta a_t k_\phi):= |t|$. In fact $\sigma(x)=d(xK, eK)$ where $d$ is the distance function on $G/K$.

Let $\what{K}=\{e_n\mid n\in\Z\}$ be the set of irreducible unitary representations of $K$, where $e_n(k_\theta)=e^{in\theta}$. 
\begin{definition} A function $f$ on $G$ is said to be of type $(n,n)$ if \begin{equation}\label{n-n-type} f(k_\theta x k_\phi)=e_n(k_\theta) f(x) e_n(k_\phi),\,\, k_\theta, k_\phi \in K,  \,x\in G.\end{equation}  \end{definition}

 Let $M=\{\pm I\}$ where $I$ is the $2\times 2$ identity matrix. The unitary dual of $M$ is $\what{M}=\{\sigma_+, \sigma_-\}$ where $\sigma_+$ is the trivial representation and $\sigma_-$ is the only nontrivial unitary irreducible representation of $M$. Let $\Z^{\sigma^+}$ (respectively $\Z^{\sigma^-}$) be the set of even (respectively odd) integers.

For $\sigma\in\what{M}$ and $\lambda\in \C$, let $\left(\pi_{\sigma,\lambda}, H_\sigma\right)$ be the principal series representation of $G$, where $H_\sigma$ is the subspace of $L^2(K)$ generated by the orthonormal set $\{e_n\mid n\in\Z^\sigma\}$ is given by $$\left(\pi_{\sigma, \lambda}(x) e_n\right)(k_\theta)=e^{-(\lambda + 1)H(x^{-1}k_\theta^{-1})} e_n\left(K(x^{-1}k_\theta^{-1})^{-1}\right).$$This representation is unitary if and only if $\lambda\in i\R$.
For every $k\in\Z^\ast$, the set of nonzero integers, there is a discrete series representation $\pi_k$, which occur as a subrepresentation of $\pi_{\sigma, |k|}, k\in \Z\setminus \Z^\sigma$. For $n\in \Z^\sigma$ and $k\in\Z\setminus \Z^\sigma$, let $$\Phi_{\sigma,\lambda}^{n,n}(x)=\langle \pi_{\sigma, \lambda}(x) e_n, e_n\rangle,$$ and $$\Psi_k^{n,n}(x)=\langle \pi_k(x) e_n^k, e_n^k\rangle_k,$$ be the matrix coefficients of the Principal series and the discrete series representations respectively, where ${e_n^k}$ are the renormalised basis and $\langle \,,\rangle_k$ is the renormalised inner product for $\pi_k$. Therefore the integral representation of $\Phi_{\sigma,\lambda}^{n,n}$ is given by $$\begin{array}{lll}
\Phi_{\sigma,\lambda}^{n,n}(x) &=& \int_K e^{-(\lambda +1)H(x^{-1}k_\theta^{-1})} e_n\left(K(x^{-1}k_\theta^{-1})^{-1}\right) e_{-n}(k_\theta)\,dk_\theta, \\ \\
&=& \int_K e^{-(\lambda +1)H(x^{-1}k_\theta)} e_n\left(K(x^{-1}k_\theta)^{-1}k_\theta\right) \,dk_\theta. \end{array}$$ This follows that $\Phi_{\sigma,\lambda}^{n,n}$  is a $(n, n)$ type function. We observe that $\Phi_{\sigma^+, \lambda}^{0,0}$ is the elementary spherical function, denoted by $\phi_\lambda$. Also, $$|\Phi_{\sigma,\lambda}^{n,n}(x)|\leq \int_K e^{-(\Re\lambda +1)H(xk_\theta)}\,dk_\theta=\phi_{\Re\lambda}(x).$$ It is well known that for $\lambda\in\C$, $$|\phi_\lambda(x)|\leq C (1 + \sigma(x)) e^{(|\Re\lambda|-1)\sigma(x)}.$$ Therefore for all $\lambda\in\C$ we have \begin{equation}\label{estimate-98}|\Phi_{\sigma,\lambda}^{n,n}(x)|\leq C (1 + \sigma(x)) e^{(|\Re\lambda|-1)\sigma(x)}.
\end{equation}

For $\sigma\in \what{M}$, let $$-\sigma=\left\{\begin{array}{lll}
\sigma^- & \text{ if } & \sigma=\sigma^+, \\ 
\sigma^+ & \text{ if } & \sigma=\sigma^-.
\end{array}\right.$$  For $k\in \Z^\ast$, let $\sigma\in \what{M}$ be determined by $k\in\Z^{-\sigma}$ defined by 
$$\Z(k)=\left\{\begin{array}{lll}
\{n\in \Z^\sigma\mid n\geq k+1\} & \text{ if } & k\geq 1, \\ \\

\{n\in \Z^\sigma\mid n\leq k-1\} & \text{ if } & k\leq -1 .
\end{array}\right.$$
Then for $k\in \Z^\ast$ and $n\in \Z(k)$, we have (\cite[Proposition 7.3]{Barker}) \begin{equation}\label{relation-phi-psi}\Phi_{\sigma, |k|}^{n,n}=\Psi_k^{n,n},
\end{equation}  where $\sigma\in\what{M}$ is determined by $k\in\Z^{-\sigma}$. 

For a $(n,n)$-type function $f$ the spherical Fourier transform of $f$ is defined by $$\what{f}_H(\sigma, \lambda)=\int_G f(x)\Phi_{\sigma, \lambda}^{n,n}(x^{-1})\, dx,$$ and $$\what{f}_B(k)=\int_G f(x) \Psi_k^{n,n}(x^{-1})\,dx,$$ where $\sigma\in\what{M}$ is determined by $n\in \Z^\sigma$ and $k\in \Z^{-\sigma}$.

For $\sigma\in\what{M}$ and $n\in\Z^\sigma$, let $$L_\sigma^{n,n}=\{k\in\Z^{-\sigma}\mid 0<k<n \text{ or } n<k<0\}.$$
Then for a nice $(n,n)$-type function $f$ the inversion formula is given by (\cite[Theorem 10.2]{Barker}):

$$f(x)=\frac{1}{4\pi^2} \int_{i\R} \what{f}_H(\sigma, \lambda) \Phi_{\sigma, \lambda}^{n,n}(x) \mu(\sigma, \lambda)\,d\lambda + \frac{1}{2\pi} \sum_{k\in L_\sigma^{n,n}}\what{f}_B(k) \Psi_k^{n,n}(x) |k|,$$ where $\sigma\in \what{M}$ is determined by $n\in \Z^\sigma$ and $\mu(\sigma, \lambda)$ is given by \begin{equation}\label{plancherel-measure-sl2}\mu(\sigma, \lambda)=\left\{ \begin{array}{lll}
\frac{\lambda\pi i}{2}\tan(\frac{\pi\lambda}{2}) & \text{ if } & \sigma=\sigma^+, \\ \\

\frac{-\lambda\pi i}{2}\cot(\frac{\pi\lambda}{2}) & \text{ if } & \sigma=\sigma^-.
\end{array}\right.
\end{equation}

Let $\mathfrak{g}=\mathrm{sl}(2, \R)$ be the Lie algebra of $\mathrm{SL}(2, \R)$ and $\mathcal U(\mathfrak g)$ be the universal envelloping algebra of $\mathfrak g$.
For $0<p\leq 2$ the $L^p$-Schwartz spaces $\mathcal C^p(G)_{n,n}$ is the set of all $(n,n)$-type functions $f\in C^\infty(G)$ such that $$\sup_{x\in G}(1 +\sigma(x))^s\phi_0(x)^{-\frac 2p}\left|f(D;x;E)\right|<\infty,$$ for any $D, E\in\mathcal U(\mathfrak g)$ and any integer $s\geq 0$, where $f(D;x;E)$ is defined by $$f(D;x;E)=\frac{d}{dt}\left|_{t=0}\right.\frac{d}{ds}\left|_{s=0}\right.f(\exp tD\,x\,\exp sE).$$
The Schwartz space $\mathcal C^p(G)_{n,n}$ is topologized by the seminorms $$\sigma_{D, E, s}^p(f)=\sup_{x\in G}(1 +|x|)^s\phi_0(x)^{-\frac 2p}\left|f(D;x;E)\right|.$$ Then it follows that $C_c^\infty(G)_{n,n}$ is dense in $\mathcal C^p(G)_{n,n}$ and $\mathcal C^p(G)_{n,n}$ is dense in $L^p(G)_{n,n}$.

Let $\mathcal C^2_B(\what{G})_{n,n}$ be the set of all functions on $L_\sigma^{n,n}$. Then we have (see \cite[Theorem 16.1]{Barker})
\begin{theorem}
The map $f \mapsto (\what{f}_H, \what{f}_B)$ is a topological isomorphism from $\mathcal C^2(G)_{n,n}$ onto $\mathcal S(i\R)_e \times \mathcal C^2_B(\what{G})_{n,n}$.
\end{theorem}

 The complexification of $\mathfrak{g}$ is $\mathfrak{g} + i \mathfrak{g}=sl(2, \C)$ is the Lie algebra of $G_\C=\mathrm{SL}(2, \C)$. The Cartan involution for $\mathfrak{g}$ is given by $\theta(A)=-A^T$.  Then $\mathfrak{g}=\mathfrak{k}\oplus\mathfrak{p}$ is the Cartan decomposition where $$\mathfrak{k}=\{A\mid \theta(A)=A\}=\{A\in \mathrm{sl}(2, \R)\mid A^T=-A\},$$ and $$\mathfrak{p}=\{A\mid \theta(A)=-A\}=\{A\in \mathrm{sl}(2, \R)\mid A^T=A\}.$$  The complexification of $\mfk$ is given by $$\mfk_\C=\mfk \oplus i\mfk=\{X\in sl(2, \C)\mid X^T=-X\}.$$ The corresponding connected group $K_\C$ whose Lie algebra is $\mfk_\C$ is given by $$K_\C=\mathrm{SO}(2, \C)=\{A\in GL(2, \C)\mid A^T A=I=AA^T\}.$$ This is noncompact group.

Now let $\mathfrak{u}=\mfk \oplus i\mfp$. Then $$\mathfrak{u}=\{A +i B\mid A \in \mfk, B\in \mfp\}=\{X\in sl(2,\C)\mid X^\ast=-X\},$$ and the corresponding connected group whose Lie algebra is $\mathfrak{u}$ is $U=\mathrm{SU}(2)$. 
Therefore starting with the Riemannian symmetric space of noncompact type $G/K=\mathrm{SL}(2, \R)/\mathrm{SO}(2)$ we get a Riemannian symmetric space of compact type $U/K=\mathrm{SU}(2)/\mathrm{SO}(2)$. Such a compact symmetric space $U/K$ is called the compact dual of $G/K$. We observe that \begin{equation}\nonumber 
G/K\cong \mathbb H^2=\{x+iy\mid x\in\R, y>0\} \text{ and } U/K\cong S^2.
\end{equation} It is easy to check that the complexification of $\mathfrak{g} $ is same as complexification of $\mathfrak{u}$, that is $\mathfrak{g}_\C=\mathfrak{u}_\C$. Also both the Riemannian symmetric spaces $G/K$ and $U/K$ are embedded as a totally real submanifold in the (non-Riemannian) symmetric space $G_\C/K_\C$.

Let $\mathcal P_m$ be the space of polynomials in two variables, with complex coefficients and homogeneous of degree $m$. We note that dimension of $\mathcal P_m$ is $m+1$. The set of irreducible representations of $U=\mathrm{SU}(2)$ is given by $\pi_m, m\in \N\cup\{0\}$ on $\mathcal P_m$ defined by $$\pi_m(g)f(u, v)=f\left((u, v) g)\right).$$

It is well known that $\pi_m$ has weights, $-m, -m +2, \cdots, m-2, m$ with weight vectors (say), $v_{-m}, v_{-m+2},  \cdots,v_{m-2}, v_m$ respectively (see \cite{Faraut}).
We say $\pi_m$ is $K$-spherical if there exists a nonzero function $f\in\mathcal P_m$ such that \begin{equation}\nonumber
 \pi_m(k_\theta)f=f \text{ for all } k_\theta\in K.
 \end{equation}

For a fixed $n\in\Z$, we say $\pi_m$ is $n$-spherical if  there exists a nonzero function $f \in\mathcal P_m$ such that \begin{equation}\nonumber
\pi_m(k_\theta)f=e_n(k_\theta^{-1})f  \text{ for all } k_\theta\in K,\end{equation} where $e_n(k_\theta)=e^{in\theta}$. Also such a vector $f$ will be called a {\em $n$-spherical vector}.

\begin{proposition}\label{n-spherical-repn-positive}
For $n\in \N\cup \{0\}$, the $n$-spherical representations are given by $\pi_{2m+n}, m\in\N\cup\{0\}$.
\end{proposition}
\begin{proof}
Let $f\in\mathcal P_l$ be a $n$-spherical vector for $\pi_l$. Since $f\in\mathcal P_l$ we can write $f$ as $$f=c_0 f_0 + c_1 f_1 + \cdots + c_l f_l,$$ where $f_j\left((u, v)\right)=u^j v^{m-j}$ and $c_j\in\C$ for $ j=0, 1,\cdots m$. Then \begin{equation}\nonumber \pi_l(k_\theta)f=e^{-in\theta} f  \text{ for all } k_\theta\in K,
\end{equation} implies that \begin{equation} f\left((u\cos\theta -v\sin\theta, u\sin\theta + v\cos\theta)\right) =e^{-in\theta}f((u, v)),
\end{equation} for all $u, v$ and $\theta$. That is for all $r, \theta, \phi$ we have,
$$f\left(r\cos(\theta +\phi), r\sin(\theta +\phi)\right)= e^{-in\theta}f\left(r\cos\phi, r\sin\phi\right).$$
This implies that,
$$f\left(\cos\theta, \sin\theta\right)=e^{-i n\theta} f((1, 0)).$$
But $f$ is a homogeneous polynomial of degree $l$. Hence \begin{equation}\nonumber f\left(r\cos\theta, r\sin\theta\right)=r^l f\left(\cos\theta, \sin\theta\right)= f((1, 0)) r^l e^{-in\theta}=f((1, 0)) r^l (\cos\theta-i\sin\theta)^n.
\end{equation} Hence $$f(u, v)= f((1,0)) r^l \left(\frac ur-i\frac vr\right)^n=f((1,0)) (u^2 +v^2)^{\frac{l-n}{2}} (u-iv)^n.$$ This $f$ is a  polynomial of degree $l$ if and only if $\frac{l-n}{2}$ is a positive intger. Hence $\pi_l$ is a $n$-spherical representation if $l=n+ 2m$ with $m\in \N\cup \{0\}$ and in this case \begin{equation}\nonumber f((u,v))=(u^2 +v^2)^{\frac{l-n}{2}} (u-iv)^n,
\end{equation} is a $n$-spherical vector.

Conversely, if $l-2m=n$, it is easy to check that \begin{equation}\nonumber f((u,v))=(u^2 +v^2)^{\frac{l-n}{2}} (u-iv)^n,\end{equation} is a $n$-spherical vector.
\end{proof}
\begin{remark}
From the proof above it follows that the dimension of $n$-spherical vectors for the representation $\pi_{n+2m}$ is $1$.
\end{remark}
The proof of the following proposition is similar to that of Proposition \ref{n-spherical-repn-positive}.
\begin{proposition}
For any negative integer $n$, the $n$-spherical representations are given by $\pi_{2m+n}, m\in\N\cup\{0\}$ such that $n + m\geq 0$. In this case \begin{equation}\nonumber 
f((u,v))=(u^2 +v^2)^{m} (u-iv)^n,
\end{equation} is a $n$-spherical vector.
\end{proposition}
A function $f$ on $U$ is said to be of type $(n,n )$ if it satisfies the same condition of (\ref{n-n-type}) with $x\in G$ is replaced by $x\in U$.

We now define the $n$-spherical function $\psi_{2m +n, n}$ (associated to $\pi_{2m+n}$) by $$\psi_{2m +n, n}(x)=\frac{1}{\|f'\|^2}\langle \pi_{2m +n}(x^{-1}) f', f'\rangle,$$ where $f'$ is a $n$-fixed vector for $\pi_{2m+n}$. It is easy to check that $\psi_{2m +n, n}(e)=1$ and $\psi_{2m +n, n}$ is a $(n, n)$-type function. 

For $f\in L^1(U)$, the Fourier coefficients of $f$ are defined by $$\what{f}(m)=\int_U f(g)\pi_m(g^{-1})\,dg, \,\, m=0,1,2,\cdots.$$

 For a function $f\in L^2(U)$, the Fourier series of $f$ is  $$f(g)=\sum_{m=0}^\infty (m+1) \text{Tr}\left(\what{f}(m)\pi_m(g)\right).$$
 
 It is now easy to check that for a $(n, n)$-type function $f$ the possible nonzero Fourier coefficients are given by $$\what{f}(2m +n)=\int_U f(g)\psi_{n+2m, n}(g^{-1})\,dg, m=0, 1, 2, \cdots.$$ 
 
 Also for a $(n,n)$-type function, $n$ nonnegative integer, the Fourier series reduces to $$f(g)=\sum_{m=0}^\infty (2m+n+1)\what{f}(2m+n)\psi_{2m+n, n}(g).$$
 
 For a $(n,n)$-type function, $n$ negative integer, the Fourier series reduces to $$f(g)=\sum_{m=[\frac{|n|}{2}]}^\infty (2m+n+1)\what{f}(2m+n)\psi_{2m+n, n}(g).$$
 
 We will need the following Polar decomposition of $U$: $$U=\mathrm{SO}(2)\, B \,\mathrm{SO}(2),$$ where \begin{equation} \nonumber 
 B=\left\{\left(\begin{array}{ll}
 e^{it} & 0 \\
 0 & e^{-it}
 \end{array}\right)\mid t\in\R\right\}.
 \end{equation} 
 
 We now have the following relation between the $n$-spherical functions on $G$ and $U$.
 \begin{theorem}\label{phi-equal-psi}
 For $m,n\in \N\cup \{0\}$, the function $\psi_{2m + n,n}$ admits a holomorphic extension to $U_\C$ and \begin{equation}\nonumber \psi_{2m+n,n}\left|_G\right.=\Phi_{\sigma, 2m+n+1}^{n,n} \text{ with } n\in \Z^{\sigma}.
 \end{equation}
 \end{theorem}
 \begin{proof}
The proof is similar to \cite[Lemma 4.6]{Ho}. Since $\pi_{n+2m}$ is a representation of the compact group $U$,  it extends holomorphically to $U_\C=\mathrm {SL}(2, \C)$. We denote the extended representation again by $\pi_{n + 2m}$. Let $n$ be fixed non-negative integer.  

We define $P_K^n$ on $\mathcal P_{n+2m}$ by $$P_K^n(f)=\int_K e_n(k_\theta)\pi_{n+ 2m}(k_\theta)f\, dk_\theta.$$ Then it is easy to check that $$\pi_{n+2m}(k_\theta) P_K^n(f)=e_n(k_\theta^{-1}) P_K^n(f),$$ that is $P_K^n(f)$ is an $n$ spherical vector.  Also $ P_K^n$ is an orthogonal projection and  self adjoint (that is, $(P_K^n)^\ast=P_K^n$). Let $f'$ be a $n$-spherical vector for $\pi_{n+2m}$. Since the space of $n$-spherical vectors is of dimension $1$, $$P_K^n(f)=\frac{1}{\|f'\|^2}\langle f, f'\rangle f'.$$ Let $g$ be the highest weight vector for $\pi_{n + 2m}$ with $\langle g, \frac{f'}{\|f'\|}\rangle=1$. We claim that: $$\psi_{n+2m, n}(x)=\frac{1}{\|f'\|}\int_Ke_n(k_\theta)\langle \pi_{n+2m}(x^{-1}k_\theta)g, f'\rangle\,dk_\theta.$$

We have $$P_K^n(g)=\frac{f'}{\|f'\|}.$$ Hence $$f'=\|f'\|\,  P_K^n(g)=\|f'\|\, \int_K e_n(k_\theta) \pi_{n + 2m}(k_\theta)g\,dk_\theta.$$This implies that $$\psi_{n+2m, n}(x)=\frac{1}{\|f'\|^2} \langle \pi_{n+2m}(x^{-1})f', f'\rangle=\frac{1}{\|f'\|}\int_Ke_n(k_\theta)\langle \pi_{n+2m}(x^{-1}k_\theta)g, f'\rangle\,dk_\theta.$$This is true for all $x\in U$ and hence true for all $x\in U_\C=\mathrm{SL}(2, \C)$.

For $x\in\mathrm{SL}(2,\R)$, $$\begin{array}{lll}
\langle \pi_{n+2m}(x^{-1}k_\theta)g, f'\rangle &=& \langle \pi_{n+2m}\left(K(x^{-1}k_\theta) a(x^{-1}k_\theta) n(x^{-1}k_\theta)\right)g, f'\rangle \\ \\
&=& \langle \pi_{n+2m}\left(a(x^{-1}k_\theta) n(x^{-1}k_\theta)\right)g, \pi_{n+2m}(K(x^{-1}k_\theta)^{-1})f'\rangle.

\end{array}$$

Since $g$ is a highest weight vector for $\pi_{n+2m}$, we have \begin{equation}\nonumber \pi_{n+2m}(n(x^{-1}k_\theta))g=g,
\end{equation} and \begin{equation}\nonumber
 \pi_{n+2m}(a(x^{-1}k_\theta))g=e^{(n+2m)H(x^{-1}k_\theta)}g.
\end{equation} Also since $f'$ is $n$-spherical vector, we have \begin{equation}\nonumber \pi_{n+2m}(K(x^{-1}k_\theta)^{-1})f'=e_n(K(x^{-1}k_\theta))f'.\end{equation} Therefore, $$\begin{array}{lll}
\langle \pi_{n+2m}(x^{-1}k_\theta)g, f'\rangle&=& \langle \pi_{n+2m}\left(a(x^{-1}k_\theta)\right)g, \pi_{n+2m}(K(x^{-1}k_\theta)^{-1})f'\rangle\\ \\
&=& \langle e^{(n+2m)H(x^{-1}k_\theta)}g, e_n(K(x^{-1}k_\theta))f'\rangle \\ \\
&=& e^{(n+2m)H(x^{-1}k_\theta)} e_{-n}(K(x^{-1}k_\theta)) (\|f'\|).
\end{array}$$

Hence for $x\in\mathrm{SL}(2, \R)$, $$\psi_{n+2m, n}(x)=\int_K e^{(n+2m)H(x^{-1}k_\theta)} e_{-n}(K(x^{-1}k_\theta)) e_n(k_\theta)\,dk_\theta=\Phi_{\sigma,n+2m+1}^{n,n}(x).$$This completes the proof.
 \end{proof}
The proof of the following theorem is similar.
 
 \begin{theorem}
 Let $n$ be a negative integer. Then for integers $m$ such that $n+m\geq 0$, $\psi_{2m + n,n}$ admits a holomorphic extension to $U_\C$ and \begin{equation}\nonumber \psi_{2m+n,n}\left|_G\right.=\Phi_{\sigma, 2m+n+1}^{n,n} \text{ with } n\in \Z^{\sigma}. \end{equation} 
 \end{theorem}
 
 We need to estimate the $n$-spherical function $\psi_{2m+n, n}$ on $U_\C$. The following decomposition will help us in this regard.
 \begin{proposition}\label{decomposition-1}
 $G_\C=\mathrm{SL}(2,\C)$ has a unique decomposition $$\mathrm{SL}(2,\C)=\mathrm{SU}(2)\exp\overline{\mathfrak{a}^+}\,\,\mathrm{SO}(2,\C),$$ where $\mathfrak{a}=\left\{\left(\begin{array}{ll}
 t & 0\\
 0 & -t
 \end{array}\right)\mid t\in\R\right\}$.
 \end{proposition}
 \begin{proof}
 Let $\theta(X)=-X^\ast$ be the Cartan involution for $\mathfrak{g}_\C=sl(2,\C)$. Then  $$\mathfrak{k_1}:=\mathfrak{g}_\C^{\theta}=\{X\in sl(2,\C)\mid \theta(X)=X\}=\{X\in sl(2,\C)\mid X^\ast=-X\}, $$  $$\mathfrak{p_1}:=\{X\in sl(2,\C)\mid d\theta(X)=-X\}=\{X\in sl(2,\C)\mid X^\ast=X\},$$ so that $$\mathfrak{g}_\C=\mathfrak{k_1}\oplus \mathfrak{p_1}.$$ Let us consider another involution $\sigma$ of $\mathfrak{g}_\C$ by $$\sigma(X)=-X^T.$$ Then $$\mathfrak{h}:=\mathfrak{g}_\C^\sigma=\{X\in sl(2,\C)\mid \sigma(X)=X\}=\{X\in sl(2,\C)\mid X^T=-X\}=so(2,\C),$$ and $$\mathfrak{q}:=\{X\in sl(2,\C)\mid \sigma(X)=-X\}=\{X\in sl(2,\C)\mid X^T=X\},$$ so that $$\mathfrak{g}_\C=\mathfrak{h}\oplus \mathfrak{q}.$$ Then $$\mathfrak{p_1}\cap \mathfrak{q}=\{X\in sl(2,\C)\mid \overline{X}^T=X=X^T\}=\{X\in sl(2,\R)\mid X^T=X\}=\mathfrak{p}_{sl(2,\R)},$$ and $$\mathfrak{p_1}\cap \mathfrak{h}=\{X\in sl(2,\C)\mid \overline{X}^T=X=-X^T\}=\left\{\left(\begin{array}{ll}
 0 & ib\\ \\
 -ib & -0
 \end{array}\right)\mid  b\in \R\right\}.$$Therefore (see \cite[Proposition 2.2, p. 106]{Heckman}), any element $g\in\mathrm{SL}(2,\C)$ can uniquely be written as $$g=k\exp X\exp Y \text{ for some } k\in\mathrm{SU}(2), X\in \mathfrak{p_1}\cap \mathfrak{q}, Y\in \mathfrak{p_1}\cap \mathfrak{h}.$$ Also since $\exp X\in \mathrm{SL}(2,\R)$, Cartan decomposition implies $$\exp X=k_\theta \exp Z\, k_\phi \text{ for some } k_\theta, k_\phi \in\mathrm{SO}(2) \text{ and for some unique }Z\in \left\{\left(\begin{array}{ll}
 t & 0\\ \\
 0 & -t
 \end{array}\right)\mid t\geq 0\right\}.$$ Hence $$g=(kk_\theta) \exp Z (k_\phi \exp Y) \text{ where } kk_\theta\in \mathrm{SU}(2), Z\in \left\{\left(\begin{array}{ll}
 t & 0\\ \\
 0 & -t
 \end{array}\right)\mid t\geq 0\right\},$$ and $ k_\phi\exp Y\in \mathrm{SO}(2,\C).$
 \end{proof}
 The following theorem gives the required estimate of the spherical function $\psi_{2m+n,n}$ on $U_\C$ (see \cite[Proposition 12]{Lassale} for corresponding result on $X_\C$.)
 \begin{proposition}\label{estimate}
 The $n$-spherical function $\psi_{2m+n,n}$ satisfies the following estimate: $$|\psi_{2m+n, n}(g)|\leq C \,e^{(2m+n)t} |e_n(h^{-1})|,$$ where $g=k\exp a(g) h$ with $k\in \mathrm{SU}(2), h\in \mathrm{SO}(2,\C)$ and $a(g)=\left(\begin{array}{ll}
 t & 0\\
 0 & -t
 \end{array}\right)\in\overline{\mathfrak{a}^+}.$
 \end{proposition}
 \begin{proof}
 The following steps will lead to the proof of the theorem.

 \noindent{\em Step-1:} Let $\pi:U\rightarrow \mathcal B(V)$ be a unitary representation of a compact group $U$. Then $d\pi:\mathfrak u\rightarrow \mathcal B(V)$ is a representation of the Lie algebra $\mathfrak u$ and since $\pi$ is unitary so $d\pi(X)^\ast=-d\pi(X)$. We extend $d\pi$ to $\mathfrak u_\C$ by, $(d\pi)^\C: \mathfrak u + i\mathfrak u\rightarrow \mathcal B(V)$ as $$(d\pi)^\C(X + iY)=d\pi(X) + i\, d\pi(Y).$$Then $$(d\pi)^\C(iY)^\ast=\left(i d\pi(Y)\right)^\ast=i\,d\pi(Y).$$ We define $\pi^\C:U_\C\rightarrow \mathcal B(V)$  as $\pi^\C(\exp(X+iY))=\exp(d\pi^\C(X + iY))$. Then \begin{equation}\label{eqn-101} \pi_\C(\exp(iY))^\ast \pi_\C(\exp(iY))=\exp d\pi^\C(iY)^\ast \exp d\pi^\C(iY)=\exp\left(2id\pi(Y)\right).
 \end{equation}
 
 \noindent{\em Step-2:} Let $f'$ be a $n$-spherical vector for $\pi_{n+2m}$ (with respect to $K$). Therefore it follows that $f'$ is a $n$-spherical vector for $\pi_{n+2m}^\C$ (with respect to $K_\C=\mathrm{SO}(2,\C)$).
 
 \noindent{\em Step-3:}  We consider an orthonormal basis of $\mathcal P_{n +2m}$ taking one element as $f^0=\frac{f'}{\|f'\|}$, where $f'$ is a $n$-spherical vector and let the basis be $\{f^0, f^1, \cdots, f^{n +2m}\}$. We claim that $$\sum_{j=0}^{2m+n}\left|\langle \pi_{n+2m}^\C(g)f^0, f^j\rangle\right|^2\leq C e^{(4m+2n)t} |e_n(h^{-1})|^2,$$ where $g=k\exp a(g) h$ and $a(g)=\left(\begin{array}{ll}
 t & 0\\
 0 & -t
 \end{array}\right)$. As a consequence it will follow that $$\left|\psi_{n+2m, n}(k\exp a(g)h)\right|\leq e^{(2m+n)t} |e_n(h^{-1})|.$$

 The projection $P_K^n$ satisfies $$\pi_{n+2m}(k) P_K^n(f)=e_n(k^{-1}) P_K^n(f).$$ Using the properties that $P_K^n$ is an orthogonal projection, self adjoint and the space of $n$-spherical vectors is of dimension $1$ it follows that
$$\begin{array}{lll}
 \sum_{j=0}^{2m+n}\left|\langle \pi_{n+2m}^\C(k\exp a(g) h)f^0, f^j\rangle\right|^2 &=& \sum_{j=0}^{2m+n}\left|\langle \pi_{n+2m}^\C(k\exp a(g)) e_n(h^{-1})f^0, f^j\rangle\right|^2 \\ \\
 &=& \sum_{i, j=0}^{2m+n}|e_n(h^{-1})|^2\left|\langle \pi_{n+2m}^\C(k\exp a(g)) P_K^n f^i, f^j\rangle\right|^2 \\ \\
 &=& |e_n(h^{-1})|^2 \left\|\pi^\C_{2m+n}\left(k\exp a(g)\right) P_K^n\right\|_{\text{HS}}^2.\\ \\
 
 \end{array}$$
 
 The expression above is equal to $$|e_n(h^{-1})|^2 \text{Tr}\left((P_K^n)^\ast \circ\pi^\C_{2m+n}\left(k\exp a(g)\right)^\ast \circ\pi^\C_{2m+n}\left(k\exp a(g)\right)\circ P_K^n\right).$$
 By Step-1 and the fact that $\pi_m$ has weights, $-m, -m +2, \cdots, m-2, m$ with weight vectors, $v_{-m}, v_{-m+2},  \cdots,v_{m-2}, v_m$ respectively we have
 
 $$\begin{array}{ll}
 \sum_{j=0}^{2m+n}\left|\langle \pi_{n+2m}^\C(k\exp a(g) h)f^0, f^j\rangle\right|^2 \\ \\
 = |e_n(h^{-1})|^2 \text{Tr}\left(P_K^n \circ \pi^\C_{2m+n}\left(\exp 2a(g)\right) \right)\\ \\
 
  = |e_n(h^{-1})|^2 \sum_{j=0}^{n+2m} \langle P_K^n\circ \pi_{n+2m}^\C(\exp 2a(g)) v_{-2m-n+2j}, v_{-2m-n+2j}\rangle\\ \\
  =|e_n(h^{-1})|^2  \sum_{j=0}^{n+2m} \langle \pi_{n+2m}^\C(\exp 2a(g)) v_{-2m-n+2j}, P_K^n v_{-2m-n+2j}\rangle\\ \\
  =|e_n(h^{-1})|^2  \sum_{j=0}^{n+2m} \langle e^{2(-2m-n+ 2j)t} v_{-2m-n+2j}, P_K^n v_{-2m-n+2j}\rangle\\ \\
  \leq |e_n(h^{-1})|^2 e^{(4m+2n)t}\, \text{Tr}( P_K^n).
 \end{array}$$
 The last inequality follows because $$\langle v_{-2m-n+2j}, P_K^n v_{-2m-n+2j}\rangle=\|P_K^n v_{-2m-n+2j}\|^2\geq 0.$$
 But $$\text{Tr}( P_K^n)=\sum_{i=0}^{n+2m}\langle P_K^n f^i, f^i\rangle=1.$$This implies that $$\sum_{j=0}^{2m+n}\left|\langle \pi_{n+2m}^\C(k\exp a(g) h)f^0, f^j\rangle\right|^2 \leq |e_n(h^{-1})|^2 e^{(4m+2n)t}.$$This completes the proof.
 \end{proof}
 \begin{corollary}\label{cor3}
 For $t\geq 0$, we have $$\left|\psi_{2m+n, n}(\exp\left(\begin{array}{rl}
 -t & 0\\
 0 & t
 \end{array}\right))\right|\leq C \,e^{(2m+n)t}.$$
 \end{corollary}
 This follows because $$\left(\begin{array}{rl}
 0& 1\\
 1 & 0
 \end{array}\right)\left(\begin{array}{rl}
 e^t & 0\\
 0 & e^{-t}
 \end{array}\right)\left(\begin{array}{rl}
 0 & 1\\
 1 & 0
 \end{array}\right)=\left(\begin{array}{ll}
 e^{-t} & 0\\
 0 & e^t
 \end{array}\right).$$
 
 We recall that \begin{equation}\nonumber
 \Phi_{\sigma,\lambda}^{n,n}(a_t) = \int_K e^{-(\lambda +1)H(a_{-t} k_\theta)} e_{-n}\left(K(a_{-t} k_\theta)\right) e_n(k_\theta) \,dk_\theta.
 \end{equation}
 We now extend holomorphically this $ \Phi_{\sigma,\lambda}^{n,n}$ to a subset of $A_\C$ (where $A_\C$ is the complexification of $A$ inside $G_\C$) containing $A$ in the following way:
 
 We have \begin{equation}\nonumber 
 a_{t+is}k_\theta=\left(\begin{array}{lll}
 
 e^{t+is}\cos\theta & e^{t+is} \sin\theta \\
 
 -e^{-t-is}\sin\theta & e^{-t-is}\cos\theta
 
 \end{array}\right).\end{equation}
 Then from (\ref{Iwasawa}) it follows that the function $t\mapsto e^{-(\lambda +1)H(a_t k_\theta)}$ extends holomorphically to a certain neighboorhod of $A$ as \begin{equation} \label{iwasawa-1}
 e^{-(\lambda +1) H(a_{t+is}k_\theta)}=\left( e^{2(t+is)} \cos^2\theta + e^{-2(t+is)}\sin^2\theta\right)^{-\frac{\lambda +1}{2}},
 \end{equation} and the function $t\mapsto e_m(K(a_t k_\theta)$ extends  holomorphically to a certain neighboorhod of $A$ as 
 \begin{equation} \label{iwasawa-2}
 e_m\left(K(a_{t+is}k_\theta)\right)=\left(\frac{e^{t+is}\cos\theta +i e^{t-is}\sin\theta}{\sqrt{e^{2(t+is)}\cos^2\theta  + e^{-2(t+is)}\sin^2\theta}}\right)^{-m}.
 \end{equation} 
 
 We now extend $\Phi_{\sigma, \lambda}^{n,n}(a_{t+is})$ in the following way:
\begin{equation}\label{extension-to-A_C}
 \Phi_{\sigma, \lambda}^{n,n}(a_{t+is})=\int_K e^{-(\lambda +1) H(a_{-t-is}k_\theta)} e_{-n} \left( K(a_{-t-is}k_\theta)\right) e_n(k_\theta) \,dk_\theta .
 \end{equation}
This function is holomorphic for $|s|<\delta_1$ for some $\delta_1>0$. The following proposition gives an estimate of $n$-spherical function on a certain neighboorhood of $A$.
 \begin{proposition}\label{estimate-phi}
For $\sigma\in\what{M}$, $$|\Phi_{\sigma,\lambda_1 + i\lambda_2 }^{n,n}(a_{t+is})|\leq C\,e^{2|n|\,|t|} e^{|\lambda_1|\,|t| + |\lambda_2|\,|s|},$$ for all $t, s\in \R$ with $|s|<\delta$ for some $\delta>0$.
 \end{proposition}
 \begin{proof}
  \noindent{\bf Case-1:} Let $n$ be a non-negative integer. We have (see \cite[(4.17)]{Koornwinder}) that for $t\in\R$ $$\Phi_{\sigma,\lambda}^{n,n}(a_t)=(\cosh t)^{2n} \phi_\lambda^{(0, 2n)}(t),$$ where $\phi_\lambda^{(0,2n)}$ is the Jacobi function given by $$\phi_\lambda^{(0,2n)}(t)= _2F_1(n+
 \frac{1-\lambda}{2}, n+
 \frac{1-\lambda}{2}, 1, -\sinh^2 t).$$ 
 
 The hypergeometric function associated to the root system $BC_1=\{-\alpha,\alpha, -2\alpha, 2\alpha\}$ is the classical Gauss hypergeometric functions $\phi_\lambda^{(a, b)}$ with $a=\frac{m_\alpha + m_{2\alpha}-1}{2}$ and $b=\frac{m_{2\alpha}-1}{2}$. It is  known by Opdam \cite[Theorem 3.15]{Opdam} that for each $\lambda\in\C$, the function $\phi_\lambda^{(0, 2n)}$ has a holomorphic extension to $\{t+is \in\C \mid  |s|<\delta_2\}$ for some $\delta_2>0$. Therefore, there exists $\delta>0$ such that for $t\in\R$ and $|s|<\delta$, $$\Phi_{\sigma,\lambda}^{n,n}(a_{t + is})=(\cosh (t+ is))^{2n} \phi_\lambda^{(0, 2n)}(t+is).$$
 
 Since, $m_{2\alpha}\geq 0$ and $m_\alpha + m_{2\alpha}\geq 0$ (i.e.  $a\geq -\frac 12, b\geq -\frac 12$) by  Ho and {\'O}lafsson (see \cite[Prop. A.6]{Ho}) we have $$\left|\phi_{\lambda_1 +i\lambda_2}^{(0,2n)}(t+is)\right|\leq C e^{|\lambda_2|\,|s| + |\lambda_1|\,|t|}.$$ Hence we have $$\left|\Phi_{\sigma, \lambda_1 + i\lambda_2}^{n,n}(a_{t+is})\right|\leq e^{2n t} e^{|\lambda_1|\,|t| + |\lambda_2|\,|s|}.$$
 
 \noindent{\bf Case-2:} Let $n$ be a negative integer. It follows from (\ref{iwasawa-1}), (\ref{iwasawa-2})) that $$\overline{e^{-(\lambda +1) H(a_{t+is}k_\theta)}}=e^{-(\overline{\lambda} +1) H(a_{t-is}k_\theta)},$$ and $$\overline{e_m\left(K(a_{t+is}k_\theta)\right)}=e_{-m}\left(K(a_{t-is}k_\theta)\right).$$
 Hence $$\overline{\Phi_{\sigma, \lambda}^{n,n}}(a_{t+is})=\Phi_{\sigma, \overline{\lambda}}^{-n,-n}(a_{t-is}).$$Therefore from Case-1, it follows that $$\left|\Phi_{\sigma, \lambda}^{n,n}(a_{t+is})\right|=\left|\overline{\Phi_{\sigma, \lambda}^{n,n}(a_{t+is})}\right|=\left|\Phi_{\sigma, \overline{\lambda}}^{-n,-n} (a_{t-is})\right|\leq e^{-2nt} e^{|\lambda_1|\,|t| + |\lambda_2|\,|s|}.$$
 \end{proof}

Let $A, p, \delta$ be real numbers such that $A<\frac{\pi}{2}, p>0$ and $0<\delta\leq 1$. Let $$\mathcal H(\delta)=\{\lambda\in\C\mid \Re\lambda>-\delta\},$$ and $$\mathcal H(A, p, \delta)=\left\{a:\mathcal H(\delta)\rightarrow\C \text{ holomorphic }\mid |a(\lambda)|\leq C e^{-p(\Re\lambda) + A|\Im\lambda|} \text{ for all }\lambda\in\mathcal H(\delta)\right\}.$$
We now present an analogue of Ramanujan's Master theorem alluded to introduction.
\begin{theorem}\label{Ramanujan-even-sl2}
Let $n$ be a positive integer and $a\in \mathcal H(A, p, \delta)$ with $A, p, \delta$ as above. Then
\begin{enumerate}
\item The $(n,n)$-type spherical Fourier series \begin{equation} \label{eqn-29}
f(x)=\sum_{m=0}^\infty (-1)^{m+1}(2m + n+1) a(2m+1+n)\psi_{n+2m, n}(x),
\end{equation} converges on a compact subset of $\Omega_p:= \mathrm{SO}(2) \left\{\left(\begin{array}{ll}
e^{it} & 0 \\
0 & e^{-it}
\end{array}\right)\mid t\in\R \text{ with } |t|<p\right\} \mathrm{SO}(2)$ and defines
a holomorphic function on a neighboorhood \begin{equation}\nonumber
\Omega_p'=\mathrm{SU}(2) \left\{\left(\begin{array}{ll}
e^{it} & 0 \\
0 & e^{-it}
\end{array}\right)\mid t\in\C \text{ with } |\Im t|<p\right\} \mathrm{SO}(2,\C), \text{ of } \Omega_p \text{ in } G_\C.
\end{equation} 

\item Let $\eta\in \R$ with $0\leq \eta<\delta$.  Then 
$$f(a_t)=\frac 12\int_{-\eta-i\infty}^{-\eta+i\infty}\left(a(\lambda) b(\lambda) + a(-\lambda) b(-\lambda)\right)\Phi_{\sigma, \lambda}^{(n,n)}(a_t)\mu(\sigma, \lambda)\,d\lambda $$ \begin{equation} \label{eqn-30} + \sum_{k\in L_\sigma^{n,n}} (-1)^{\frac{k-n-1}{2}} k \,a(k)\Psi_{k} (a_t), \end{equation} for $t\in\R$ with $|t|<p$, where $b(\lambda)$ is defined by 
\begin{equation}\nonumber
b(\lambda)\mu(\sigma, \lambda) =\left\{ \begin{array}{lll}
\left(\frac{i}{4}(-1)^{-\frac n2 +1}\right)\frac{\lambda}{\cos\frac{\pi\lambda}{2}} &  \text{ if }  n \text{ is even, } \\  
 \left(\frac{i}{4}(-1)^{-[\frac n2] +1}\right)\frac{\lambda}{\sin\frac{\pi\lambda}{2}} & \text{ if }   n \text{ is odd.}
\end{array}\right.
\end{equation}
The integral above is independent of $\eta$ and converges uniformly on compact subsets of $\{a_t\mid t\in\R\}$ and extends as a $(n,n)$-type function on a neighboorhood of $G$ in $G_\C$.

\item The extension of $f$ to $G$ satisfies the following:  \begin{equation}\nonumber
\frac{1}{4\pi^2}\int_{G}f(x)\Phi_{\sigma, \lambda}^{n,n}(x^{-1})\,dx=\frac 12\left( a(\lambda) b(\lambda) + a(-\lambda) b(-\lambda)\right),\hspace{.1in} \lambda\in i\R,
\end{equation} and for $k\in L_\sigma^{n,n}$
 \begin{equation}\nonumber
 \frac{1}{2\pi}\int_{G}f(x)\Psi_{k}^{n,n}(x^{-1})\,dx =(-1)^{\frac{k-n-1}{2}} a(k).
 \end{equation} 
\end{enumerate}
\end{theorem}

To prove the theorem we need the following elementary estimates:
\begin{lemma}\label{lemma-1}
\begin{enumerate}
\item $|\cos \frac{\pi}{2}(s + ir)| \geq e^{\frac{\pi}{2}r}$ if $r$ is away from $0$.
\item $|\cos\frac{\pi}{2}(2k + is)|\geq e^{\frac{\pi}{2}s}$ for all $s$.
\item $|\cos \frac{\pi}{2}(s - i2k)| \geq e^{k\pi}$ for $k\geq 1$.
\item On a small neighboorhood of $i\R$, $|\cos \frac{\pi}{2}(s + ir)| \geq e^{\frac{\pi}{2}|r|}$ if $r$ is away from $0$.
\end{enumerate}
\end{lemma}

\begin{proof}[Proof of Theorem \ref{Ramanujan-even-sl2}:]
We first assume that $n$ is an even positive integer. In this case, proof of $(1)$ follows straight away from the Corollary \ref{cor3}.

For the proof of $(2)$, we note that the function $b(\lambda)\mu(\sigma^+, \lambda)=c\frac{\lambda}{\cos\frac{\pi\lambda}{2}}$, where $c=\left(\frac{i}{4}(-1)^{-\frac n2 +1}\right)$ has simple poles at $\lambda=\pm 1, \pm 3, \pm 5, \cdots$ and it is easy to check that $$\text{Res}_{\lambda=2j+1} \left(b(\lambda) \mu(\sigma^+, \lambda)\right)=(-1)^{j+1} \frac{2c}{\pi}(2j+1).$$Also it is easy to see from Lemma \ref{lemma-1} that for $\lambda$ on imaginary axis, \begin{equation}\nonumber
 |b(\lambda)\mu(\sigma^+, \lambda)|\leq |\lambda| e^{-\frac{\pi}{2}|\Im\lambda|}. 
\end{equation}
It follows from (\ref{plancherel-measure-sl2}), that for $\lambda\in\C$ 
 \begin{equation} \nonumber b(\lambda)=\frac{-2ci}{\pi}\frac{1}{\sin(\frac{\pi\lambda}{2})}.\end{equation} Clearly this function has simple poles at $\lambda=0, \pm 2, \pm 4, \cdots$. 

We now consider the integral 
\begin{equation}\nonumber 
I=\int_{-i\infty}^{+i\infty} a(\lambda) b(\lambda)\Phi_{\sigma^+, \lambda}^{n,n}(a_t)\mu(\sigma^+, \lambda)\,d\lambda.\end{equation}  By choosing the rectangular path $\gamma_k$ joining the vertices $(0,  2ki), (2k, 2ki), (2k, -2ki), (0, -2ki)$ and observing that the poles of the integrand are situated at the points  $\lambda=1, 3, 5, \cdots, 2k-1$ inside the path $\gamma_k$ it follows that \vspace{-.08in}$$\begin{array}{ll} & \int_{\gamma_r} a(\lambda) b(\lambda)\Phi_{\sigma^+, \lambda}^{(n,n)}(a_t)\mu(\sigma^+, \lambda)\,d\lambda\\
 =& 2\pi i \sum_{m=0}^{k-1} a(2m+1)\Phi_{\sigma^+, 2m+1}^{(n,n)} (a_t)\text{ Res }_{\lambda = 2m+1}\left(b(\lambda) \mu(\sigma^+, \lambda)\right) \\ 
=& 2\pi i\sum_{m=0}^{k-1} (-1)^{m+1} \frac{2c}{\pi}(2m+1) \,a(2m+1)\Phi_{\sigma^+, 2m+1}^{(n,n)} (a_t) .
\end{array}$$

 For $t\in\R$ with $|t|<p$, we have $$\begin{array}{ll}
\int_{2ki}^{2k+2ki} |a(\lambda)|\, |b(\lambda)|\,|\Phi_{\sigma^+, \lambda}^{(n,n)}(a_t)|\mu(\sigma^+, \lambda)\,d\lambda\\ \\  
=\int_0^{2k} |a(s+2ki)|\,|\Phi_{\sigma^+, s+2ki}^{(n,n)}(a_t)|\, |b(s+2ki)|\,\mu(\sigma^+, s+2ki)\,ds & \\ \\
\leq \int_0^{2k} e^{-ps + A2k} e^{(s-1)t} (s+2k) e^{-k\pi}\,ds \, (\text{because of } (\ref{estimate-98}) \text{ and Lemma } \ref{lemma-1})\\ \\
\leq C k^2 e^{(A-\frac{\pi}{2})2k}.
\end{array}$$
This last quantity goes to $0$ as $k\rightarrow \infty$.

We now consider the integral \begin{equation}\label{eqn-91} \int_{2k + i2k}^{2k-i2k} |a(\lambda)|\, |b(\lambda)|\,|\Phi_{\sigma^+, \lambda}^{(n,n)}(a_t)|\mu(\sigma^+, \lambda)\,d\lambda.
\end{equation}
 Let $t\geq 0$ with $t<p$.  Now,  $$\begin{array}{ll}
\int_{2k - i2k}^{2k+i2k} |a(\lambda)|\, |b(\lambda)|\,|\Phi_{\sigma^+, \lambda}^{(n,n)}(a_t)|\mu(\sigma^+, \lambda)\,d\lambda\\ \\  
=\int_{-2k}^{2k} |a(2k +is )|\,|\Phi_{\sigma^+, 2k+is}^{(n,n)}(a_t)|\, |b(2k +is)|\,\mu(\sigma^+, 2k +is)\,ds & \\ \\
\leq \int_{-2k}^{2k} e^{-2pk + A|s|} e^{(2k-1)t} (2k + |s|) e^{-\frac{\pi}{2}|s|}\,ds  \,(\text{because of } (\ref{estimate-98}) \text{ and Lemma } \ref{lemma-1})\\ \\
=2\int_{0}^{2k} e^{-2pk + A|s|} e^{(2k-1)t} (2k + |s|) e^{-\frac{\pi}{2}|s|}\,ds\\ \\
\leq C k e^{-2kt}\int_{0}^{2k}  e^{(A-\frac \pi 2)|s|}\,ds.
\end{array}$$
The last quantity goes to $0$ as $k\rightarrow \infty$. Hence the integral in (\ref{eqn-91}) goes to zero as $k\rightarrow\infty$, for $t\in\R$ with $|t|<p$.

Next,  $$\begin{array}{lll} 
\int_{2k - i2k}^{-i2k} |a(\lambda)|\, |b(\lambda)|\,|\Phi_{\sigma^+, \lambda}^{(n,n)}(a_t)|\mu(\sigma^+, \lambda)\,d\lambda \\ \\
 
=\int_{2k}^{0} |a(s-i2k)|\, |b(s-i2k)|\,|\Phi_{\sigma^+, s-i2k}^{(n,n)}(a_t)|\mu(\sigma^+, s-i2k)\,ds \\ \\
   
\leq \int_{2k}^{0} e^{-ps + A 2k} e^{(s-1)t}(|s| + 2k) e^{-k\pi}\,ds  \,(\text{because of } (\ref{estimate-98}) \text{ and Lemma } \ref{lemma-1})\\ \\
\leq  C k e^{(A-\frac{\pi}{2})2k}.
\end{array}$$
This last quantity goes to zero as $k$ goes to infinity. Hence we have 
$$ I=2\pi i\sum_{m=0}^{\infty} (-1)^{m+1} \frac{2c}{\pi}(2m+1) \,a(2m+1)\Phi_{\sigma^+, 2m+1}^{(n,n)} (a_t).$$
We now write $$I=I_1 + I_2,$$ where $$\begin{array}{lll} I_1 &=& 2\pi i\sum_{m=0}^{\frac{n-2}{2}} (-1)^{m+1} \frac{2c}{\pi}(2m+1) \,a(2m+1)\Phi_{\sigma^+, 2m+1}^{(n,n)} (a_t),\\ \\
I_2 & =& 2\pi i\sum_{m=\frac{n}{2}}^{\infty} (-1)^{m+1} \frac{2c}{\pi}(2m+1) \,a(2m+1)\Phi_{\sigma^+, 2m+1}^{(n,n)} (a_t).\end{array}$$ 
We know that for $m\leq \frac{n-2}{2}$ that is, for $2m+1< n$ (see (\ref{relation-phi-psi})), $$\Phi_{\sigma^+, 2m+1}^{(n,n)}(a_t)=\Psi_{2m+1}(a_t).$$ Therefore $$\begin{array}{lll} I_1&=& 2\pi i\sum_{m=0}^{\frac{n-2}{2}} (-1)^{m+1} \frac{2c}{\pi}(2m+1) \,a(2m+1)\Psi_{2m+1}(a_t)\\ \\
&=& \sum_{m=0}^{\frac{n-2}{2}} (-1)^{m+1} 4ci (2m+1) \,a(2m+1)\Psi_{2m+1}(a_t).\end{array}$$
We recall from Theorem \ref{phi-equal-psi} that for $m\geq 0$,  $$\Phi_{\sigma^+, n+ 2m+1}^{(n,n)}(a_t)=\psi_{n+2m,n}(a_t).$$Therefore $$\begin{array}{lll}
I_2&=& 2\pi i\sum_{m=0}^{\infty} (-1)^{m+1+\frac n2} \frac{2c}{\pi}(2m+1+n) \,a(n+2m+1)\psi_{n+2m,n}(a_t)\\ \\
&=& \sum_{m=0}^{\infty} (-1)^{m+1+\frac n2} 4 ci (2m+1+n) \,a(n+2m+1)\psi_{n+2m,n}(a_t).
\end{array}$$
We have choosen $c$ such that $4ci (-1)^{\frac{n}{2}}=1$. Therefore, for $t\in\R$ with $|t|<p$, $$f(a_t)=\frac 12\int_{-i\infty}^{+i\infty}\left(a(\lambda) b(\lambda) + a(-\lambda) b(-\lambda)\right)\Phi_{\sigma^+, \lambda}^{(n,n)}(a_t)\mu(\sigma^+, \lambda)\,d\lambda $$ $$+ \sum_{j=0}^{\frac{n-2}{2}} (-1)^{j- \frac n2} (2j+1) a(2j +1)\,\Psi_{2j+1} (a_t).$$  From Proposition \ref{estimate-phi}, Lemma \ref{lemma-1} and using the fact that $A<\frac{\pi}{2}$, it is clear that
the integral $$\int_{-i\infty}^{+i\infty}\left(a(\lambda) b(\lambda) + a(-\lambda) b(-\lambda)\right)\Phi_{\sigma^+, \lambda}^{(n,n)}(a_t)\mu(\sigma^+, \lambda)\,d\lambda,$$ exists and holomorphic for all $t\in\C$ with $|\Im t|<\delta_1$ for some $\delta_1>0$.  Therefore using (\ref{relation-phi-psi}) and (\ref{extension-to-A_C}) it follows that $f(a_t)$ has a holomorphic extension for all $t\in\C$ with $|\Im t|<\delta_2$ for some $\delta_2>0$. Hence $$f(a_t)=\frac 12\int_{-i\infty}^{+i\infty}\left(a(\lambda) b(\lambda) + a(-\lambda) b(-\lambda)\right)\Phi_{\sigma^+, \lambda}^{(n,n)}(a_t)\mu(\sigma^+, \lambda)\,d\lambda $$ $$+ \sum_{j=0}^{\frac{n-2}{2}} (-1)^{j- \frac n2} (2j+1) a(2j +1)\,\Psi_{2j+1} (a_t),$$ for $t\in\C$, $|\Im t|<\delta_2$.

Using Cauchy integral formula and using the estimates of $\Phi_{\sigma, \lambda}^{n,n}, b(\lambda)\mu(\sigma, \lambda)$ and  $a(\lambda)$ we can prove  that for $0<\eta<\delta (\leq 1)$, $$\begin{array}{ll}\frac 12\int_{-i\infty}^{+i\infty}\left(a(\lambda) b(\lambda) + a(-\lambda) b(-\lambda)\right)\Phi_{\sigma^+, \lambda}^{(n,n)}(a_t)\mu(\sigma^+, \lambda)\,d\lambda \\ \\
=\frac 12\int_{-\eta-i\infty}^{-\eta+i\infty}\left(a(\lambda) b(\lambda) + a(-\lambda) b(-\lambda)\right)\Phi_{\sigma^+, \lambda}^{(n,n)}(a_t)\mu(\sigma^+, \lambda)\,d\lambda.\end{array} $$

For the proof of $(3)$, we first observe that  $a(\lambda) b(\lambda) + a(-\lambda) b(-\lambda)$ is holomorphic in $|\Re \lambda|<\delta$, as $a(\lambda) b(\lambda) + a(-\lambda) b(-\lambda)=\left(a(\lambda)-a(-\lambda)\right) b(\lambda)$ and $a(\lambda)-a(-\lambda)$ has a zero at $\lambda=0$ and $a(\lambda)$ is holomorphic there. Now for $s$ away from $0$, $$|a(r+is) b(r+is)|\leq C e^{-pr + A|s|} e^{-\frac{\pi}{2}|s|}.$$ This estimate, together with Cauchy integral formula shows that \begin{equation}\nonumber
\lambda\mapsto a(\lambda) b(\lambda) + a(-\lambda) b(-\lambda) \in \mathcal S(i\R)_e.\end{equation} Therefore, using the inversion formula we get the desired formula. This completes the proof for the case when $n$ is an even nonnegative integer.

Proof for the case when $n$ is odd is similar.
\end{proof}

We now state the theorem for $n$ negative integer and we observe that there is no discrete series representation involves in this case.
\begin{theorem}
Let $n$ be a negative integer and $a\in \mathcal H(A, p, \delta)$. Then\\
\begin{enumerate}
\item The spherical Fourier series \begin{equation}\nonumber
 f(x)=\sum_{m=\frac {|n|}{2}}^\infty (-1)^{m+1} a(2m+1+n)\psi_{n+2m, n}(x),\end{equation} converges on a compact subset of $\Omega_p:= \mathrm{SO}(2) \left\{\left(\begin{array}{ll}
e^{it} & 0 \\
0 & e^{-it}
\end{array}\right)\mid t\in\R \text{ with } |t|<p\right\} \mathrm{SO}(2)$ \\
and defines a  holomorphic function on a neighboorhood \begin{equation}\nonumber \Omega_p'=\mathrm{SU}(2) \left\{\left(\begin{array}{ll}
e^{it} & 0 \\
0 & e^{-it}
\end{array}\right)\mid t\in\C \text{ with } |t|<p\right\} \mathrm{SO}(2, \C) ,\text{ of } \Omega_p \text{ in } X_\C.
\end{equation}

\item Let $\eta\in \R$ with $0\leq \eta<\delta$. Then $$f(a_t)=\frac 12\int_{-\eta-i\infty}^{-\eta+i\infty}\left(a(\lambda) b(\lambda) + a(-\lambda) b(-\lambda)\right)\Phi_{\sigma, \lambda}^{(n,n)}(a_t)\mu(\sigma, \lambda)\,d\lambda, $$  for $t\in\R$ with $|t|<p$, where $b(\lambda)$ is defined by 
\begin{equation} \nonumber 
b(\lambda)\mu(\sigma, \lambda)=\left\{ \begin{array}{lll}
\left(\frac{i}{4}(-1)^{n +1}\right)\frac{\lambda}{\cos\frac{\pi\lambda}{2}} & \text{ if } n \text{ is even, } \\  
 \left(\frac{i}{4}(-1)^{-(n +1}\right)\frac{\lambda}{\sin\frac{\pi\lambda}{2}} & \text{ if } n  \text{ is odd.}
 \end{array}\right.
\end{equation}
The integral above is independent of $\eta$ and converges uniformly on compact subsets of $\{a_t\mid t\in\R\}$ and extends as a $(n,n)$-type function on a neighboorhood of $G$ in $G_\C$.

\item The extension of $f$ to $G$ satisfies the following:  \begin{equation}\frac{1}{4\pi^2}\int_{G}f(x)\Phi_{\sigma^+, \lambda}^{n,n}(x^{-1})\,dx=\frac{1}{2}\left( a(\lambda) b(\lambda) + a(-\lambda) b(-\lambda)\right), \lambda\in i\R.
\end{equation}

\end{enumerate}
\end{theorem}

\begin{remark}
It can be easily seen that the decay condition on the function $a\in \mathcal H(A, p, \delta)$ (that is, $A<\frac{\pi}{2}$) is optimum. In deed, Theorem \ref{Ramanujan-even-sl2} is not true for $a(\lambda)$, where  $$a(\lambda)=\left\{\begin{array}{lll}
 \lambda \sin\frac{\pi\lambda}{2} & \text{ if } n \text{ is odd}, \\ \\  
\lambda \cos\frac{\pi\lambda}{2}  &\text{ if } n \text{ is even.}
\end{array}\right.$$
Because for the case $n$ even, it follows from (\ref{eqn-29}) that $f(x)=0$ on $\Omega_p'$. In particular $f(I)=0$ where $I$ is the identity matrix. Therefore the integral in the right side of (\ref{eqn-30}) should vanishes at $t=0$ but from the expression of $a(\lambda)$ given above it follows that the integral in the right side of (\ref{eqn-30}) does not exist.  
The reason for the case $n$ odd is similar. It can be easily checked that $A<\pi$ is optimum in \cite{Bertram, Olafsson-1}. This difference occurs due to the parmetrization of $\mathfrak{a}, \mathfrak{a^\ast}$.
\end{remark}

\section{Ramujan's Master theorem for $\chi_l$-radial functions on $\mathrm{SU(1, n)}$}
In this section we shall prove an analogue of Ramanujan's Master theorem for the $\chi_l$-radial functions on $\mathrm{SU}(1, n)$. We refer to \cite{Ho} and references therein for the preliminaries.

Let $G=\mathrm{SU}(1, n)=\left\{g\in \mathrm{GL}_n(\C)\mid g^\ast I_{1, n}g=I_{1, n}, \det g=1\right\}$ where $g^\ast$ is the complex conjugate and \begin{equation}\nonumber 
I_{1,n}=\left(\begin{array}{lll}
1 & 0_{1\times n}\\ 
0_{n\times 1} & -I_{n\times n}

\end{array}\right).
\end{equation}Then the Lie algebra $\mathfrak{g}$ of $G$ is given by  $$\mathfrak{g}=\left\{\left(\begin{array}{rll}
-\text{Tr}(C) & A \\
A^\ast & C_{n\times n}
\end{array}\right)\mid C^\ast_{n\times n}=-C_{n\times n}\right\}.$$ If $\theta(X)=-X^\ast$ denotes the Cartan involution on $\mathfrak{g}$, then $$\mathfrak{k}=\mathfrak{g}^\theta=\left\{X\in\mathfrak{g}\mid \theta(X)=X\right\}=\left\{\left(\begin{array}{rll}
-\text{Tr}(C) & 0 \\
0 & C
\end{array}\right)\mid C^\ast=-C\right\}.$$
Then  $\mathfrak{k}$ is the Lie algebra of the compact group $$K=\left\{\left(\begin{array}{rll}
(\det U)^{-1} & 0 \\
0 & U
\end{array}\right)\mid U\in \mathrm{U}(n)\right\}=\mathrm{S}\left(\mathrm{U}(1)\times \mathrm{U}(n)\right).$$For $$\mathfrak{p}=\left\{X\in \mathfrak{g}\mid \theta(X)=-X\right\}=\left\{\left(\begin{array}{lll}
0 & B \\
B^\ast & 0

\end{array}\right)\mid B\in M_{1\times n}(\C)\right\},$$  we have \begin{equation}
\nonumber
\mathfrak{g}=\mathfrak{k}\oplus \mathfrak{p}.
\end{equation}
Let \begin{equation}\label{eqn-5}
H=\left(\begin{array}{lll}
0 & 0_{1\times n-1} & 1_{1\times 1} \\
0_{n-1\times 1} & 0_{n-1\times n-1} & 0_{n-1\times 1} \\
1_{1\times 1} & 0_{1\times n-1} & 0_{1\times 1}
\end{array}\right)\in \mathfrak{p},
\end{equation} and let $\mathfrak {a}=\R H$. Then $\mathfrak {a}$ is a maximal abelian subspace of $\mathfrak{p}$. Let $\beta\in \mathfrak{a}^\ast$  be such that $\beta(H)=1$. Then \begin{equation}\nonumber
\Sigma(\mathfrak{g}, \mathfrak{a})=\{\beta, 2\beta\},
\end{equation}  $m_\beta:=\dim \mathfrak{g}_\beta=2(n-1), m_{2\beta}:=\dim\mathfrak{g}_{2\beta}=1$ and $\rho=\frac 12(m_\beta + 2m_{2\beta})=n$.

Let \begin{equation} \nonumber 
\mathfrak{u}=\mathfrak{k}\oplus i\mathfrak{p}=\left\{\left(\begin{array}{rll}
-\text{Tr}(C) & iB \\
iB^\ast & C
\end{array}\right)\mid C^\ast=-C, B\in M_{1, n}(\C)\right\}. 
\end{equation} This is the set of $(n+1)\times (n+1)$  skew hermitian matrices with trace $0$. Then  $\mathfrak{u}$ is Lie algebra of the compact group $U=\mathrm{SU}(n+1)$. It follows that $U=\mathrm{SU}(n+1)$ is the compact dual of $G=\mathrm{SU}(1, n)$ and $\mathfrak{g}_\C=\mathfrak{u}_\C$. Let $G_\C$ be the analytic subgroup of $\mathrm{GL}(n+1, \C)$ with Lie algebra $\mathfrak{g}_\C$. We define, $U_\C=G_\C$.

All one dimensional representations of $K$ are parametrized by $l\in\Z$ and given by $$\chi_l\left(\left(\begin{array}{rll}
(\det U)^{-1} & 0 \\
0 & U
\end{array}\right)\right)= (\det U)^l.$$
\begin{definition}\label{defn-chi_l-radial}
A function $f$ on $G$ (respectively on $U$) is said to be $\chi_l$-radial if \begin{equation} \nonumber 
f(k_1 x k_2)=\chi_l(k_1^{-1}) f(x) \chi_l (k_2^{-1}), \, k_1, k_2\in K,
\end{equation}
$x\in G$ (respectively for  $x\in U$).
\end{definition}
A representation $\pi$ of $U$ on a Hilbert space $V$ is said to be spherical if  \begin{equation}
\nonumber
V_\pi^K:=\left\{v\in V\mid \pi(k)v=v \text{ for all } k\in K\right\},
\end{equation}
is non empty and in this case $\dim V_\pi^K=1$. It is known that the spherical representations of $U$ are parametrized by $2m, m\in \N\cup \{0\}$. Analogously, a representation $\pi$ of $U$ on a Hilbert space $V$ is said to be $\chi_l$-spherical if  \begin{equation}
\nonumber
V_\pi^l:=\left\{v\in V\mid \pi(k)v=\chi_l(k) v \text{ for all } k\in K\right\},
\end{equation}
is non empty. It then follows that $\dim V_\pi^l=1$. It is known that the $\chi_l$ spherical representations of $U$ are parametrized by $2m + |l|, m\in \N\cup \{0\}$. Let the representation space for $\pi_{2m+ |l|}$ be $V_{2m +|l|}$. Also let $\lambda_1, \lambda_2, \cdots, \lambda_r$ be the weights for $\pi_{2m + |l|}$ with weight vectors say $v_1, v_2, \cdots, v_r$ respectively. Then \begin{equation} \lambda_i(H)\leq 2m +|l|,  i=1, 2, \cdots, r,
\end{equation}  as $2m + |l|$ is the highest weight.

For fixed $m\in \N\cup\{0\}, l\in\Z$, let $f_0$ be the unique $\chi_l$-spherical vector for $\pi_{2m+|l|}$ with $\|f_0\|=1$. We define the elementary spherical function of type $\chi_l$ on $U$ by $$\psi_{2m + |l|}(x):=\langle \pi_{2m+|l|}(x)f_0, f_0\rangle.$$ Clearly this function is $\chi_l$-radial on $U$. Let $\mathfrak{n}=\mathfrak{g}_\beta \oplus \mathfrak{g}_{2\beta}$ and let $N$ be the analytic subgroup of $G$ with Lie algebra $\mathfrak{n}$ and let $A=\exp \mathfrak{a}$. Then we have the Iwasawa decomposition $G=KAN$ and any element of $g\in G$ can be uniquely written as \begin{equation}\nonumber 
g=k a_t n \text{ for some } k\in K, t\in \R \text{ and } n\in N.
\end{equation} Here $a_t=\exp(tH)$ and we define $K(g)=k, H(g)=t$. We also have the following Cartan decomposition: $G=K \overline{A^+} K$ where $\overline{A^+}=\{a_t\mid t\geq 0\}$.

The elementary spherical function of type $\chi_l$ on $G$ with spectral parameter $\lambda$ is defined by $$\phi_{\lambda, l}(g)=\int_K e^{-(\lambda +\rho)H(gk)}\chi_l(K(gk)^{-1}k)\,dk.$$ The function $\phi_{\lambda, l}$ satisfies the following relations:
\begin{enumerate}
\item $\phi_{\lambda, l}=\phi_{\mu, l}$ if and only if  $\lambda=\pm \mu$.
\item $\phi_{\lambda, l}(g)=\phi_{-\lambda, -l}(g^{-1})=\phi_{\lambda, -l}(g^{-1})=\phi_{-\lambda, l}(g)$.
\end{enumerate}
The  function $\phi_{\lambda, l}$ which is analytic on $G$ extends to a holomorphic function on $G_\C$ and by restriction to $U$ gives an elementary spherical function on $U$ if and only if $\lambda\in \pm(2\N\cup \{0\} + n + |l|)$. That is, \begin{equation} \phi_{2m + |l| + n, l}\left|\right. _U=\psi_{2m + |l|}\, m\in \N\cup \{0\}.
\end{equation}

For a $\chi_l$-radial function $f$ on $U$, the Fourier series is defined by $$f(x)=\sum_{m=0}^\infty d_{\pi_{2m + |l|}} \what{f}(2m +|l|)\psi_{2m +|l|}(x),$$ where $d_{\pi_{2m + |l|}}$ is the dimension of $V_{2m +|l|}$ and $$\what{f}(2m +|l|)=\int_U f(x)\psi_{2m +|l|}(x^{-1})\,dx.$$

For a $\chi_l$-type radial function $f$ on $G$, the spherical Fourier transform  of $f$ is defined by $$\what{f}(\lambda)=\int_G f(x)\phi_{\lambda, l}(x^{-1})\,dx.$$

 From now on we will assume that $|l|<n$.
We define $$c(\lambda, m)= 2^{\rho-\lambda} \frac{\Gamma(\frac{m_\beta + m_{2\beta} +1}{2}) \Gamma(\lambda)}{\Gamma(\frac{\lambda}{2}+\frac{m_\beta}{4} +\frac{m_{2\beta}}{2}) \Gamma(\frac{\lambda}{2} +\frac{m_\beta}{4} +\frac{1}{2})}.$$ It follows that $c(\rho, m)=1$.
We define a new multiplicity function $m_+(l)$ on $\Sigma$ by \begin{equation}
m_+(l)=(2(n-1)-2|l|, 1 + 2|l|).
\end{equation} 
Then \begin{equation} \label{c-function-new-multiplicity} c(\lambda, m_+(l))= 2^{n + |l|-\lambda} \frac{\Gamma(n) \Gamma(\lambda)}{\Gamma(\frac{\lambda}{2}+\frac{n + |l|}{2} )\Gamma(\frac{\lambda}{2} +\frac{n-|l|}{2})},
\end{equation} and $\rho(m_+(l))=n + |l|$.

For a $\chi_l$ radial function $f$ on $G$, the following inversion formula holds: \begin{equation}\label{inversion-chi_l}
f(x)=\frac{1}{2 \cdot 4^{n |l|}}\int_{i\R} \what{f}(\lambda)\phi_{\lambda, l}(x) |c(\lambda, m_+(l))|^{-2} \,d\lambda.
\end{equation}
We now have the following proposition which is an analogue of Proposition \ref{decomposition-1}.
\begin{proposition}
Any $x\in G_\C$ has a unique decomposition $x=u \exp tH k$ for some $u\in \mathrm{SU}(n+1), k\in K_\C$, and $t\geq 0$ and $H$ is as (\ref{eqn-5}).

\end{proposition}
\begin{proof}Let $\mathfrak{g}$ be the Lie algebra of $G$ and $K$ is the maximal compact subgroup of $G$ with Lie algebra of $K$ is $\mathfrak{k}$.
Let $\theta(X)=-X^\ast$ be the Cartan involution for $\mathfrak{g}_\C$. Then $$\mathfrak{g}_\C= \mathfrak{k}_1\oplus \mathfrak{p}_1,$$ where $$\mathfrak{k}_1= \{X\in \mathfrak{g}_\C\mid \theta(X)=X\}=\{X\in \mathfrak{g}_\C\mid X^\ast=-X\},$$ and $$\mathfrak{p}_1= \{X\in \mathfrak{g}_\C\mid \theta(X)=-X\}=\{X\in \mathfrak{g}_\C\mid X^\ast=X\}.$$ Let $K_1$ be the analytic subgroup whose Lie algebra is $\mathfrak{k}_1$. Therefore $K_1=\mathrm{SU}(n +1)$.
Let $\sigma:\mathfrak{g}_\C\rightarrow \mathfrak{g}_\C$ be an involution given by $\sigma(X)=-X^T$. Then again $$\mathfrak{g}_\C= \mathfrak{h}\oplus \mathfrak{q},$$ where $$\mathfrak{h}= \{X\in \mathfrak{g}_\C\mid \sigma(X)=X\}=\{X\in \mathfrak{g}_\C\mid X^T=-X\},$$ and $$\mathfrak{q}= \{X\in \mathfrak{g}_\C\mid \sigma(X)=-X\}=\{X\in \mathfrak{g}_\C\mid X^T=X\}.$$

Therefore (see \cite[Proposition 2.2, p. 106]{Heckman}), any element $g\in G_\C$ can uniquely be written as $$g=u\exp X\exp Y, \,  \text{ for some } u\in K_1, X\in \mathfrak{p}\cap \mathfrak{q}, Y\in \mathfrak{p}\cap \mathfrak{h}.$$

It is now easy to check that $Y\in \mathfrak{p}\cap \mathfrak{h}$ implies that $Y$ is purely imaginary $(n+1)\times (n+1)$ skew symmetric matrix and hence $Y\in \mathfrak{k}_\C$. Also $X\in \mathfrak{p}\cap \mathfrak{q}$ implies that $X$ is real $(n+1)\times (n+1)$ symmetric matrix. Hence $X\in \mathfrak{p}_{\mathfrak{g}}$ and therefore by Cartan decomposition we have $\exp X=k_1\exp tH k_2$ for some $k_1, k_2\in K$ and $t\geq 0$.

Hence we have $$g=u k_1 \exp tH k_2\exp Y,$$ where $uk_1\in U$ and $ k_2\exp Y\in K_\C$.

\end{proof}
We now prove the required estimate of $\psi_{2m +|l|}$ which is an analogue of Theorem \ref{estimate}. 
\begin{theorem}\label{estimate-1}
For $x\in G_\C$, we have $$|\psi_{2m + |l|}(x)| \leq C e^{(2m + |l|)t} |\chi_l(k^{-1})|,$$ where $x= u \exp tH k, u\in U, t\geq 0, k\in K_\C$.
\end{theorem}
\begin{proof}
Let $f_0$ be the $\chi_l$ spherical vector for $\pi_{2m +|l|}$, that is, $$\pi_{2m +|l|}(k)f_0=\chi_l(k) f_0.$$Then it is easy to check that $$\pi_{2m +|l|}^\C(k_1)f_0=\chi_l(k_1) f_0 \text{ for all } k_1\in K_\C,$$ where $\pi_{2m +|l|}^\C$ is the complexification of $\pi_{2m +|l|}$. Let $\{f_0, f_1, \cdots, f_r\}$ be an orthonormal basis of $V_{2m + |l|}$ such that $f_0$ is $\chi_l$ spherical vector for $\pi_{2m +|l|}$ and $\|f_0\|=1$. We claim that $$\sum_{j=0}^r \left| \langle \pi_{2m +|l|}^\C(x) f_0, f_j\rangle\right|^2 \leq C  |\chi_l(k^{-1})|^2 e^{2t(2m +|l|)},$$ where $x= u \exp tH k, u\in U, t\geq 0, k\in K_\C$ and as a consequence we get $$|\psi_{2m +|l|}(x)|\leq C e^{(2m + |l|)t} |\chi_l(k^{-1})|.$$

We define $P_l$ on $V_{2m + |l|}$ by $$P_l(f)(x)=\int_K \chi_l(k)\pi_{2m +|l|}(kx)f \,dk.$$ It is an orthogonal projection and $P_l^\ast=P_l$. Now $$\begin{array}{lll}
 \sum_{j=0}^{r}\left|\langle \pi_{2m+ |l|}^\C(u\exp tH k)f_0, f_j\rangle\right|^2 &=& |\chi_l(k^{-1})|^2 \sum_{j=0}^r \left|\langle \pi_{2m +|l|}^\C(u\exp tH))f_0, f_j\rangle\right|^2 \\ \\
 &=&  |\chi_l(k^{-1})|^2 \sum_{i, j=0}^r \left|\langle \pi_{2m + |l|}^\C(u\exp tH) P_l f_i, f_j\rangle\right|^2 \\ \\
 &=&  |\chi_l(k^{-1})|^2  \left\|\pi^\C_{2m+|l|}\left(u\exp tH\right) P_l\right\|_{\text{HS}}^2.\\ \\
 
 \end{array}$$
 
 The expression above is equal to \begin{equation}\label{eqn-100} |\chi_l(k^{-1})|^2 \text{Tr}\left(P_l^\ast \circ\pi^\C_{2m+|l|}\left(u\exp tH\right)^\ast \circ\pi^\C_{2m+|l|}\left(u\exp tH\right)\circ P_l\right).
 \end{equation} It can be shown using (\ref{eqn-101}) that \begin{equation}\label{eqn-102}
 \pi^\C_{2m+|l|}(\exp tH)^\ast \circ \pi^\C_{2m+|l|}(\exp tH)=\pi^\C_{2m+|l|}(\exp 2tH).
 \end{equation} Therefore it follows from (\ref{eqn-100}), (\ref{eqn-102}) that
 $$\begin{array}{ll}
 \sum_{j=0}^r \left|\langle \pi_{2m + |l|}^\C(u\exp tH k)f_0, f_j\rangle\right|^2 
 =  |\chi_l(k^{-1})|^2 \text{Tr}\left(P_l\circ \pi^\C_{2m+|l|}\left(\exp 2H\right) \right).
  \end{array}$$
  
Recall that, $v_0, \cdots, v_r$ are the weight vectors with weights $\lambda_0, \cdots \lambda_r$ respectively and let $\lambda_r=2m + |l|$ be the highest weight. It then follows that \begin{equation} \nonumber \lambda_i(H)\leq 2m + |l|,  0\leq i\leq r.
\end{equation} Hence we have, 
  $$\begin{array}{lll}
 \sum_{j=0}^r \left|\langle \pi_{2m + |l|}^\C(u\exp tH k)f_0, f_j\rangle\right|^2 
  &=&  |\chi_l(k^{-1})|^2\sum_{i=0}^r \langle P_l\circ \pi_{2m+ |l|}^\C(\exp 2H) v_i, v_i\rangle\\ \\
 &=&  |\chi_l(k^{-1})|^2\sum_{i=0}^r \langle e^{2t \lambda_i(H)} v_i, P_l^\ast v_i\rangle.
  \end{array}$$
  
 Since, $\langle v_i, P_l^\ast v_i\rangle=\|P_lv_i\|^2$ is non negative we have, 
  $$\begin{array}{lll}
 \sum_{j=0}^r \left|\langle \pi_{2m + |l|}^\C(u\exp tH k)f_0, f_j\rangle\right|^2
  &\leq & |\chi_l(k^{-1})|^2 e^{2t(2m +|l|)}  \sum_{i=0}^r \langle  v_i, P_l^\ast v_i\rangle\\ \\
  &=&  |\chi_l(k^{-1})|^2 e^{2t(2m +|l|)} \text{Tr}(P_l).
 \end{array}$$
 But $$\text{Tr}( P_l)=\sum_{i=0}^r \langle P_l  f_i, f_i\rangle= \langle P_l  f_0, f_0\rangle=\langle f_0, f_0\rangle=1.$$ This implies that $$\sum_{j=0}^r\left|\langle \pi_{2m + |l|}^\C(u\exp tH k)f_0, f_j\rangle\right|^2 \leq |\chi_l(k^{-1})|^2 e^{2t(2m +|l|)}.$$This completes the proof.
\end{proof}

The following Proposition gives an estimate of $\chi_l$-spherical function on a certain neighboorhood of $A$ in $G_\C$.
\begin{proposition}\label{estimate-phi-l}$($see \cite[Prop. A.3, Prop. A.6]{Ho}$)$
There exists $\delta>0$ such that, the function \begin{equation}\nonumber 
t+is\mapsto \phi_{\lambda_1 + i\lambda_2, l }(\exp (t+is)H),
\end{equation} is holomorphic for all  $t, s\in \R$ with $|s|<\delta$ and satisfies the following estimates $$|\phi_{\lambda_1 + i\lambda_2, l }(\exp (t+is)H)|\leq C\,e^{|l|\,|t|} e^{|\lambda_1|\,|t| + |\lambda_2|\,|s|},$$ in that domain. 
 \end{proposition}
 
It follows from (\ref{c-function-new-multiplicity}) that
 \begin{equation} \nonumber
  \hspace{-2.5in}\frac{1}{c(\lambda, m_+(l))c(-\lambda, m_+(l))}
  \end{equation}
  
   \begin{equation}\label{product-c}=2^{-2(n + |l|)}(\Gamma(n))^{-2} \frac{\Gamma(\frac{\lambda + n + |l|}{2} )\, \Gamma(\frac{\lambda + n-|l|}{2})\, \Gamma(\frac{-\lambda + n + |l|}{2} )\,\Gamma(\frac{-\lambda + n-|l|}{2})}{\Gamma(\lambda)\Gamma(-\lambda)}.
 \end{equation}
 Using the following formulas 
\begin{equation}\nonumber
\Gamma(2z)=(\pi)^{-\frac12}2^{2z-1}\Gamma(z)\Gamma(z + \frac12), \, \Gamma(z)\Gamma(-z)=-\frac{\pi}{z\sin\pi z}, \, \Gamma(z+ \frac 12)\Gamma(-z +\frac12)=\frac{\pi}{\cos\pi z},
\end{equation} 
and (\ref{product-c}) it is easy to check the following:
\begin{equation}\label{product-c-1}
\frac{1}{c(\lambda, m_+(l))c(-\lambda, m_+(l))}=\gamma_{n,l} \,\,p_{n,l}(\lambda) \,q_{n,l}(\lambda),
\end{equation}
where $\gamma_{n,l}$ is a constant given by 
\begin{equation}\nonumber
\gamma_{n,l}=\left\{\begin{array}{lll}
- \pi\,2^{-4n -2 |l| +3)}(\Gamma(n))^{-2} & \text{ if } l \text{ is even}, \\
\pi\,2^{-4n -2 |l| +3)}(\Gamma(n))^{-2} & \text{ if } l \text{ is odd},
\end{array} \right.
\end{equation} 
 and $p_{n,l}$ is a polynomial in $\lambda$ given by 
 \begin{equation}\nonumber
 p_{n,l}(\lambda) =\left\{ \begin{array}{lll}
 
  \lambda\prod_{j=1}^{\frac{n-|l|-1}{2}}\left(\lambda^2-(n-|l|-2j)^2\right) \, \prod_{j=1}^{\frac{n+|l|-1}{2}}\left(\lambda^2-(n+ |l|-2j)^2\right) & \text{ if }  n \text{ odd}, l \text{ even},\\ \\
 
 \lambda^3  \prod_{j=1}^{\frac{n+|l|}{2}-1}\left(\lambda^2-(n+|l|-2j)^2\right) \prod_{j=1}^{\frac{n-|l|}{2}-1}\left(\lambda^2-(n- |l|-2j)^2\right) & \text{ if }  n \text{ even}, l \text{ even}, \\ \\
 
 \lambda^3  \prod_{j=1}^{\frac{n+|l|}{2}-1}\left(\lambda^2-(n+|l|-1-2j)^2\right) \prod_{j=1}^{\frac{n-|l|}{2}}\left(\lambda^2-(n- |l|-2j)^2\right) & \text{ if }  n \text{ odd}, l \text{ odd}, \\ \\

 \lambda \prod_{j=1}^{\frac{n+|l|-1}{2}}\left(\lambda^2-(n+|l|-2j)^2\right) \prod_{j=1}^{\frac{n-|l|-3}{2}}\left(\lambda^2-(n- |l|-2j)^2\right) & \text{ if }  n \text{ even}, l \text{ odd},
 \end{array}\right.
 \end{equation}
 
 and $q_{n,l}$ is given by 
 
 \begin{equation}\nonumber
 q_{n,l}(\lambda) =\left\{ \begin{array}{lll}
 
 \tan\left(\frac{\pi \lambda}{2}\right) & \text{ if }  n \text{ odd}, l \text{ even or if }n \text{ even, } l \text{ odd},\\ \\
 \cot\left(\frac{\pi \lambda}{2}\right) & \text{ if }  n \text{ even}, l \text{ even or if } n\text{ odd, }l\text{ odd}.
  \end{array}\right.
 \end{equation}

We  also have the following relation between the dimension $d_{\pi_{2m + |l|}}$ of $V_{\pi_{2m +|l|}}$ and the $c$-function ( \cite[Prop. 5.2.10, p. 78]{Heckman})
\begin{equation} \label{eqn-12}
d_{\pi_{2m + |l|}}=\alpha \frac{c(\rho(m_+), m_+(l)) c(-\rho(m_+), m_+(l))}{c(2m + \rho(m_+), m_+(l))c(-2m-\rho(m_+), m_+(l))},
\end{equation} where $\alpha$ is a fixed constant. We recall that $\rho(m_+)=n + |l|$.
 
This follows from the relation  (see \cite{Heckman} for unexplained notation) $$\frac{\wtilde{c}^*(-\rho(m_+), m_+(l))}{\wtilde{c}^*(-\mu-\rho(m_+), m_+(l))}=\frac{\wtilde{c}(-\rho(m_+), m_+(l))}{\wtilde{c}(-\mu-\rho(m_+), m_+(l))}, \,\, \mu\in P^+.$$

Let us consider the case $n$ odd, $l$ even. In this case it follows from (\ref{product-c-1}) and (\ref{eqn-12})  that $$\begin{array}{lll} 
d_{\pi_{2m + |l|}}&=& \alpha\,\, \frac{p_{n, l}(2m + n +|l|)}{p_{n, l}(n +|l|)}\frac{q_{n, l}(2m +n + |l|)}{q_{n,l}(n +|l|)}\\ \\
&=& \alpha \,\,\frac{p_{n, l}(2m + n + |l|)}{p_{n, l}(n + |l|)}.

\end{array}$$

 We are now in a position to state our final theorem.
 
\begin{theorem}\label{ramanujan-hyperbolic}
Let $n$ odd, $l$ even and $\mathcal H(A, p, \delta)$ as in Theorem \ref{Ramanujan-even-sl2}. Suppose $a\in\mathcal H(A, p,\delta)$ and let $b$ be the meromorphic function on $\C$ defined by $$\frac{b(\lambda)}{c(\lambda, m_+(l))c(-\lambda, m_+(l))}=\left(\frac{-i}{4}\right)\alpha\,\, \frac{p_{n,l}(\lambda)}{p_{n, l}(n +|l|)}\frac{1}{\sin\frac{\pi}{2}(\lambda-n-|l|)},$$ where $\alpha$ is a constant given in (\ref{eqn-12}).
Then \begin{enumerate}
\item The $\chi_l$-spherical Fourier series $$f(x)=\sum_{m=0}^\infty (-1)^m d_{\pi_{2m +l}} a(2m+|l|+n)\psi_{2m +|l|}(x),$$ converges on a compact subset of $\Omega_p:=K \left\{\exp(tH)\mid t\in\R, |t|<p\right\} K$ and defines a holomorphic function on a neighboorhood \begin{equation}\nonumber 
\Omega_p':=U \left\{\exp(tH)\mid t\in\C, |\Re t|<p\right\} K_\C, \text{ of } \Omega_p \text{ in } G_\C.
\end{equation}

\item Let $\eta\in\R$ with $0\leq \eta<n\delta$. Then
$$f(a_t)=\int_{-\eta - i\infty}^{-\eta + i\infty} \left(a(\lambda) b(\lambda) + a(-\lambda) b(-\lambda)\right)\phi_{\lambda, l}(a_t) |c(\lambda, m_+(l))|^{-2} \,d\lambda,$$  for $t\in\R$ with $|t|<p$.

The integral above converges uniformly on compact subsets of $\{a_t\mid t\in\R\}$ and extends as a $\chi_l$-radial function on a neighboorhood of $G$ in $G_\C$.

\item The extension of $f$ to $G$ satisfies: $$\frac {1}{4^l 2}\int_{G}f(x)\phi_{\lambda, l}(x^{-1})\,dx= \left( a(\lambda) b(\lambda) + a(-\lambda) b(-\lambda)\right), \lambda\in i\R.$$
\end{enumerate}
\end{theorem} 
 The following elementary estimates will be used in the proof.
\begin{lemma}\label{lemma-2}
\begin{enumerate}
\item $|\sin \frac{\pi}{2}(s + ir)| \geq e^{\frac{\pi}{2}r}$ if $r$ is away from zero.
\item $|\sin\frac{\pi}{2}(2k + is)|\geq e^{\frac{\pi}{2}s}$ for all $s$.
\item $|\sin \frac{\pi}{2}(s - i2k)| \geq e^{k\pi}$ for $k\geq 1$.
\item On a small neighboorhood of $i\R$, $|\sin \frac{\pi}{2}(s + ir)| \geq e^{\frac{\pi}{2}|r|}$ if $r$ is away from zero.
\end{enumerate}
\end{lemma}

\begin{proof}[Proof of Theorem \ref{ramanujan-hyperbolic}:]
We first observe that the function $$\frac{b(\lambda)}{c(\lambda, m_+(l))c(-\lambda, m_+(l))}=\left(\frac{-i}{4}\right)\alpha\,\, \frac{p_{n, l}(\lambda)}{p_{n, l}(n +|l|)}\frac{1}{\sin\frac{\pi}{2}(\lambda-n-|l|)},$$ has simple poles at $\lambda=n +l, n+l+2, \cdots $ in the domain $\{z\in\C\mid \Re z\geq -n\delta\}$.

From (\ref{product-c-1}) it follows that $$b(\lambda)=\left(\frac{-i}{4}\right) (-1)^{\frac{n-l+3}{2}}\left(\frac{\alpha}{\pi}\right) 2^{4n + 2|l|-3}(\Gamma(n))^2 \frac{1}{p_{n,l}(n+|1|)} \frac{1}{\sin\left(\frac{\pi\lambda}{2}\right)}.$$This is an odd function and has a simple pole at $\lambda=0$. Therefore the function \begin{equation}\nonumber \lambda\mapsto a(\lambda) b(\lambda) + a(-\lambda) b(-\lambda),
\end{equation} is holomorphic in a small neighborhood of $i\R$. Using the same argument as in the proof of the Theorem \ref{Ramanujan-even-sl2} it follows that 
\begin{equation}\nonumber \lambda\mapsto a(\lambda) b(\lambda) + a(-\lambda) b(-\lambda)\in \mathcal S(i\R).
\end{equation} Therefore $(3)$ will follow from $(2)$ by the inversion formula (\ref{inversion-chi_l}).
Rest of the proof is similar to the proof of Theorem \ref{Ramanujan-even-sl2} using Cauchy's formula and Lemma \ref{lemma-2}.
 
\end{proof}

\begin{remark}
 We define  $b(\lambda)$ in the other cases by:
\begin{equation}
\frac{b(\lambda)}{c(\lambda, m_+(l))c(-\lambda, m_+(l))}=\left\{ \begin{array}{lll}\left(\frac{-i}{4}\right)(-\alpha)\,\, \frac{p_{n,l}(\lambda)}{p_{n, l}(n +|l|)}\frac{1}{\sin\frac{\pi}{2}(\lambda-n-|l|)} & \text{ if } & n \text{ even, } l \text{ even, } \\ \\
\left(\frac{-i}{4}\right)\alpha\,\, \frac{p_{n,l}(\lambda)}{p_{n, l}(n +|l|)}\frac{1}{\sin\frac{\pi}{2}(\lambda-n-|l|)} & \text{ if } & n \text{ odd, } l \text{ even,} \\ \\
\left(\frac{-i}{4}\right)(-\alpha)\,\, \frac{p_{n,l}(\lambda)}{p_{n, l}(n +|l|)}\frac{1}{\sin\frac{\pi}{2}(\lambda-n-|l|)} & \text{ if } & n \text{ odd, } l \text{ odd}.
\end{array} \right.
\end{equation}
With this $b(\lambda)$, analogue of Theorem \ref{ramanujan-hyperbolic} holds true for the  cases above and the proof will be exactly same.  It is also easy to see that the constant $A<\frac{\pi}{2}$ is optimum.

\end{remark}

\end{document}